\documentclass[10pt]{article}
\usepackage[a4paper,margin=3cm]{geometry}
\usepackage[english]{babel}
\usepackage{amsmath}
\usepackage{amsthm, mathtools, amssymb}
\usepackage{amsfonts}
\usepackage{graphicx}
\usepackage{epstopdf}
\usepackage{stmaryrd}
\usepackage{tikz-cd}
\tikzcdset{ arrow style= math font }
\usepackage[hidelinks]{hyperref}
\usepackage{enumitem}
\usepackage{cite}
\usepackage{bm}
\usepackage[export]{adjustbox}
\numberwithin{equation}{section}
\usepackage{mathrsfs} 
\usepackage{cleveref} 
\usepackage{pgfplots}
\usepackage{tikz}

\newcommand{\tx}[1]{\makebox[\widthof{$H(\Omega^3, \nabla \times)$}]{$#1$}} 
\newcommand{\Bbbd}{\textrm{\fontencoding{U}\fontfamily{bbold}\selectfont d}}
\newcommand{\Tr}{{\operatorname{Tr}}}

\newcommand{\myd}{\text{d}}

\newcommand{\vertiii}[1]{{\left\vert\kern-0.25ex\left\vert\kern-0.25ex\left\vert #1 \right\vert\kern-0.25ex\right\vert\kern-0.25ex\right\vert}}

\newtheorem{theorem}{Theorem}[section]
\newtheorem{corollary}{Corollary}[section]
\newtheorem{remark}{Remark}[section]
\newtheorem{lemma}{Lemma}[section]

\tikzcdset{ arrow style=tikz, diagrams={>={Straight Barb[scale=1]}} }

\title{Mixed-Dimensional Auxiliary Space Preconditioners}
\author{Ana Budi\v{s}a\footnotemark[1] , Wietse M.~Boon\footnotemark[2] , Xiaozhe Hu\footnotemark[3] }
\date{}

\begin{document}
\maketitle

\footnotetext[1]{Department of Mathematics, University of Bergen, Allegaten 41, P. O. Box 7803, N-5020 Bergen, Norway.
	\href{mailto:Ana.Budisa@uib.no}{Ana.Budisa@uib.no}
}
\footnotetext[2]{Department of Mathematics, KTH Royal Institute of Technology, Lindstedtsv\"{a}gen 25, 10044 Stockholm, Sweden.
}
\footnotetext[3]{Department of Mathematics, Tufts University, 503 Boston Ave, Medford, MA 02155, USA.
}

\begin{abstract}
	This work introduces nodal auxiliary space preconditioners for discretizations of mixed-dimensional partial differential equations. We first consider the continuous setting and generalize the regular decomposition to this setting. With the use of conforming mixed finite element spaces, we then expand these results to the discrete case and obtain a decomposition in terms of nodal Lagrange elements. In turn, nodal preconditioners are proposed analogous to the auxiliary space preconditioners of Hiptmair and Xu \cite{hiptmair:xu}. Numerical experiments show the performance of this preconditioner in the context of flow in fractured porous media.
\end{abstract}

\section{Introduction} 
\label{sec:introduction}
In recent work~\cite{boon:md2018,nordbotten2017modeling, boon:flow2018}, exterior calculus and its finite element discretization has been extended to the mixed-dimensional geometries.  More precisely, for an $n$-dimensional domain, sub-manifolds of dimension $n-1$ and their intersections of dimension $n-2$, $n-3$, and so on are considered.   Suitable spaces of alternating $k$-forms are introduced and equipped with proper inner products and norms.  Based on well-defined differential operators and codifferential operators, a de Rham complex for the mixed-dimensional geometry is proposed as well.  Such a generalization of fixed-dimensional finite element exterior calculus~\cite{arnold:falk:winther_FEEC} provides a unified theoretical framework for mixed-dimensional partial differential equations (PDEs) as well as their finite element discretizations. This has wide applications in mathematical modeling and simulation, e.g., shells, membranes, fractures, and geological formations~\cite{bear2012hydraulics,ciarlet1997mathematical,nordbotten2011geological}.  

One important result in the fixed-dimensional finite element exterior calculus is the stable regular decomposition and its discrete variant~\cite{hiptmair2002finite,hiptmair:xu}.  Understanding the stable regular decompositions is at the heart of designing robust preconditioners for solving $H(\nabla \times)$- and $H(\nabla \cdot)$-elliptic problems based on the auxiliary space preconditioning framework~\cite{nepomnyaschikh1992decomposition,xu1996auxiliary}.  Based on the discrete regular decomposition, preconditioners for $H(\nabla \times)$- and $H(\nabla \cdot)$-elliptic problems can be developed, which consists of solving several $H(\nabla)$-elliptic problems and simple smoothing steps in the original space. Numerical results~\cite{KolevVassilevski2012,KolevVassilevski2009} have shown the effectiveness of such preconditioners. 

In this work, we extend the stable regular decomposition to the mixed-dimensional geometries.  Unlike the fixed-dimensional case, where the stable regular decomposition is usually derived based on the corresponding regular inverse, in the mixed-dimensional setting, we construct the regular decomposition directly by establishing such a regular decomposition on each individual sub-manifold and then combining them together properly.  Discrete regular decomposition is also generalized to the mixed-dimensional geometries.  The construction of the discrete version is similar with the fixed-dimensional counterpart.  The resulting discrete regular decomposition also involves an extra high-frequency term comparing with the stable regular decomposition as expected.  Based on discrete regular decomposition and auxiliary space preconditioning framework, we are able to develop robust preconditioners for solving abstract model mixed-dimensional PDEs~\eqref{eq:md_problem}.

In order to demonstrate the effectiveness of the proposed auxiliary space preconditioner for solving mixed-dimensional PDEs, we consider flow in fractured porous media as an example, which is modeled by Darcy's law and conservation of mass in the mixed-dimensional setting. After discretization, robust block preconditioners are designed based on the well-posedness of the discrete PDEs based on the framework developed in~\cite{loghin:wathen,mardal:winther}.  The mixed-dimensional auxiliary space preconditioner is used to invert one of the diagonal blocks in the block preconditioners.  The effectiveness of the preconditioners are verified both theoretically and numerically. 

The rest of the paper is organized as follows.  \Cref{sec:preliminaries} introduces the mixed-dimensional geometries and function spaces.  Mixed-dimensional regular decomposition is derived in \Cref{sec:regular_decomposition} and the discrete version is proposed in \Cref{sec:discretization}.  In \Cref{sec:preconditioner}, we describe the mixed-dimensional auxiliary space preconditioner for abstract mixed-dimensional PDEs and an example, flow in fractured porous media, is introduced in \Cref{sec:practical example}.  Numerical results are shown in \Cref{sec:examples} to demonstrate the robustness and effectiveness of the proposed preconditioners, and the conclusions are given in \Cref{sec:conclusion}.

\section{Preliminaries}
\label{sec:preliminaries}

In this section, we first introduce the definition of a mixed-dimensional geometry and the conventions used when referring to certain structures. Next, we summarize the relevant concepts from functional analysis for the fixed-dimensional case as well as the generalization to the mixed-dimensional setting. For a more rigorous and detailed exposition of these results, we refer the readers to \cite{boon:md2018}.

\subsection{Geometry}
\label{sub:geometry}

Given a contractible Lipschitz domain $Y \subset \mathbb{R}^n$ with $n \le 3$. Within $Y$, we introduce disjoint manifolds $\Omega_i^{d_i}$ with $i$ being the index from a global set $I$ and $d_i$ being the dimension. The superscript generally is omitted. Let $I^d$ be the subset of $I$ containing all indices $i$ with $d_i = d$. 

We refer to the union of all manifolds $\Omega_i$ as the mixed-dimensional geometry $\Omega$ and denote the subset of $d$-manifolds as $\Omega^d$, i.e.
\begin{align*}
	\Omega &:= \bigcup_{i \in I} \Omega_i, & 
	\Omega^d &:= \bigcup_{i \in I^d} \Omega_i.
\end{align*}
For each $\Omega_i$ with $i \in I$, we form a connection to each lower-dimensional manifold that coincides with (a portion of) its boundary. Each of these connections is endowed with a unique index $j$. Then, let $i_j$ be the index of the lower-dimensional manifold such that $\Omega_{i_j} \subseteq \partial \Omega_i$. We denote $\partial_j \Omega_i$ as the corresponding boundary of dimension $d_j := d_{i_j}$.

For each $i \in I$ and $d < d_i$, we define $I_i^d$ as the set of indices $j$ such that $\partial_j \Omega_i$ is $d$-dimensional. Moreover, let $I_i$ contain all indices $j$ for which $\partial_j \Omega_i$ is not empty:
\begin{align*}
	I_i^d &:= \left\{ j :\ \exists i_j \in I^d \text{ such that } \partial \Omega_i \cap \Omega_{i_j} \ne \emptyset \right\}, &
	I_i &:= \bigcup_{d = 0}^{d_i -1} I_i^d.
\end{align*} 

To exemplify a mixed-dimensional geometry $\Omega$, let us consider Figure~\ref{fig:figure1} and its corresponding index sets $I_i^d$. In this case, we have $I_2^1 = \{ 12, 13 \}$ with $i_{12} = 5$ and $i_{13} = 3$ and $I_2^0 = 16$ with $i_{16} = 6$. We note that two distinct portions of the boundary $\partial \Omega_1$ coincide geometrically with $\Omega_4$. This is represented by the two indices $\{ 9, 10 \} \subset I_1^1$ with $i_9 = i_{10} = 4$.

\begin{figure}[ht]
	\centering
	\includegraphics[width = \textwidth]{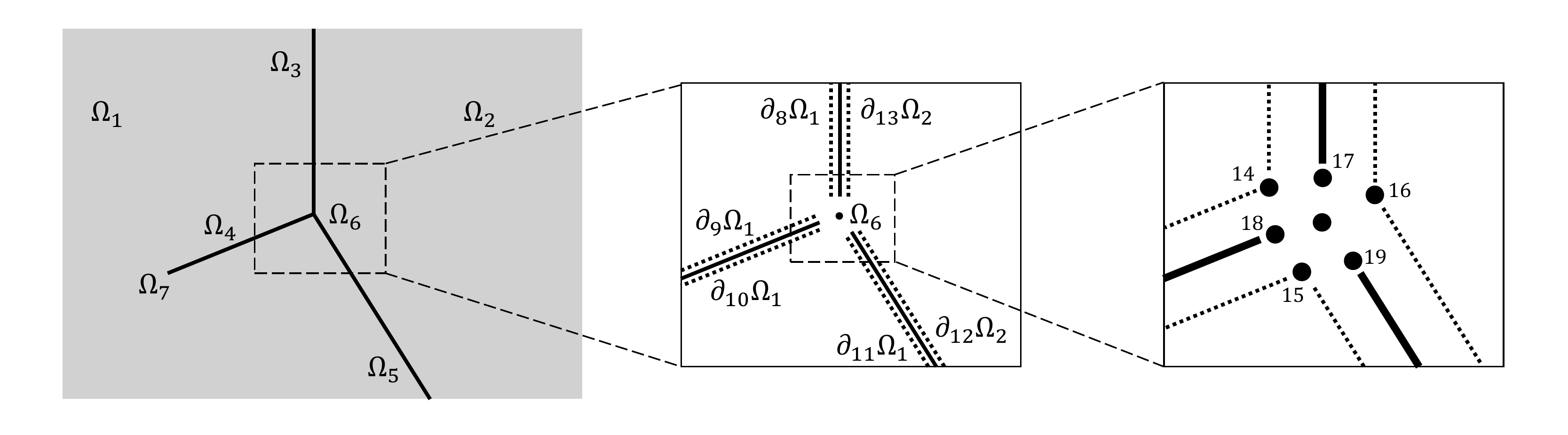}
	\caption{Example of a mixed-dimensional geometry with $n = 2$. On the left, $\Omega_i$ labels each $d_i$-manifold with $i \in I$. The index sets $I^d$ for this geometry are as follows. $I^2 = \{ 1, 2\}$ represents the 2-manifolds, $I^1 = \{ 3, 4, 5\}$ denotes the 1-manifolds and the 0-manifolds in this geometry have indices $i \in I^0 = \{ 6, 7 \}$. The middle of the figure illustrates the enumeration of boundaries $\partial_j \Omega_i$ for $i \in I^2$ and $j \in I_i^1$. On the right, the indices $j$ are shown with the property $i_j = 6$.}
	\label{fig:figure1}
\end{figure}

\subsection{Function Spaces}
\label{sub:function_spaces}

The next step is to define a function $\mathfrak{a}$ on the mixed-dimensional geometry. We do this using the language of exterior calculus \cite{spivak2018calculus}. We first introduce local function spaces on each subdomain which form the building blocks for the mixed-dimensional generalization.

For $i \in I$, let $\Lambda^k(\Omega_i)$ denote the space of differential $k$-forms on $\Omega_i$. Let $L^2 \Lambda^k(\Omega_i)$ denote the space of square integrable $k$-forms and $H \Lambda^k(\Omega_i)$ denote its subspace of forms with square integrable differential. In other words, let
\begin{align*}
	L^2 \Lambda^k(\Omega_i) 
	&:= \{ a_i \in \Lambda^k(\Omega_i) :\ \| a_i \|_{L^2(\Omega_i)} < \infty \},
	\\
	H \Lambda^k(\Omega_i) 
	&:= \{ a_i \in L^2 \Lambda^k(\Omega_i) :\ \myd a_i \in L^2 \Lambda^{k + 1}(\Omega_i) \}.
\end{align*}

With the exterior derivative $\myd$, the spaces $H \Lambda^k$ form a cochain complex, known as the de Rham complex:
\begin{equation*}
	\begin{tikzcd}
	H \Lambda^0(\Omega_i) \arrow[r,"\myd"]
	& H \Lambda^1(\Omega_i) \arrow[r,"\myd"]
	& \cdots \arrow[r,"\myd"]
	& H \Lambda^{d_i - 1}(\Omega_i) \arrow[r,"\myd"]
	& H \Lambda^{d_i}(\Omega_i).
	\end{tikzcd}
\end{equation*}
We often use the correspondence of this complex to conventional Sobolev spaces. For $d_i = 3$, this representation of the de Rham complex is given by
\begin{equation} \label{eq: FD de Rham represent}
	\begin{tikzcd}
	H(\nabla, \Omega_i) \arrow[r,"\nabla"]
	& H(\nabla \times, \Omega_i) \arrow[r,"\nabla \times"]
	& H(\nabla \cdot, \Omega_i) \arrow[r,"\nabla \cdot"]
	& L^2(\Omega_i).
	\end{tikzcd}
\end{equation}
Here, $L^2(\Omega_i)$ is the space of square-integrable functions on $\Omega_i$ and the space $H(\nabla, \Omega_i)$ is its subspace of functions with square-integrable gradients, typically denoted by $H^1(\Omega_i)$. The spaces $H(\nabla \times, \Omega_i)$ and $H(\nabla \cdot, \Omega_i)$ are defined analogously.

We use the local spaces $ L^2 \Lambda^k $ and $H \Lambda^k$ to introduce the Sobolev spaces containing mixed-dimensional differential $k$-forms on $\Omega$. For brevity, we omit the reference to the geometry and define
\begin{align*}
	L^2 \mathfrak{L}^k
	&:= \prod_{i \in I} \{ a_i \in L^2 \Lambda^{k_i}(\Omega_i): 
	\Tr_j a_i \in L^2 \Lambda^{k_i}(\Omega_{i_j}), 
	\ \forall j \in I_i \}, \\
	H \mathfrak{L}^k
	&:= \prod_{i \in I} \{ a_i \in H \Lambda^{k_i}(\Omega_i): 
	\Tr_j a_i \in H \Lambda^{k_i}(\Omega_{i_j}), 
	\ \forall j \in I_i \},
\end{align*}
with $k_i := d_i - (n - k)$. Here, and in the following, we interpret $\Lambda^k(\Omega_i)$ as zero for $k <0$ and $k > d_i$. 
Thus, we emphasize that $H \mathfrak{L}^k$ is zero on manifolds $\Omega_i$ with $d_i < n - k$. The operator $\Tr_j$ is a trace operator that restricts a form $a_i$ to $\partial_j \Omega_i$. We emphasize that for a given $\mathfrak{a} = (a_i)_{i \in I} \in H \mathfrak{L}^k$, the component $a_i$ has a well-defined trace on each $\partial_j \Omega_i$ for $j \in I_i$ with $k_i \le d_j < d_i$ by definition. 

The Gothic font is used to denote a mixed-dimensional differential form $\mathfrak{a} \in H \mathfrak{L}^k$ and we revert to classic fonts with a subscript $i$ to denote its component defined on $\Omega_i$. In the analysis, we often use the corresponding restriction operator $\iota_i$ defined such that	
\begin{align*}
	\iota_i \mathfrak{a} = a_i.
\end{align*}

Next, we define the jump operator $\Bbbd: H \mathfrak{L}^k \mapsto H \mathfrak{L}^{k + 1}$. For each $i \in I$, let
\begin{align*}
	\iota_i \Bbbd \mathfrak{a} = (-1)^{n - k} \sum_{l \in I^{d_i + 1}} \sum_{\{j \in I_{l} :\ i_j = i \}} \Tr_j a_l.   
\end{align*}
For more details on the definition of $\Bbbd$, we refer to \cite{boon:md2018}. The mixed-dimensional differential $\mathfrak{d}$ is formed as the sum of $\Bbbd$ and the exterior derivative $\myd$ such that
\begin{align*}
	\iota_i \mathfrak{da} 
	&= \myd a_i + \iota_i \Bbbd \mathfrak{a}, &
	i &\in I.
\end{align*}

We introduce the following norms for $\mathfrak{a} \in H \mathfrak{L}^k$:
\begin{align*}
	\| \mathfrak{a} \|_{L^2 \mathfrak{L}^k}^2 &:= 
	\sum_{i \in I} \| a_i \|_{L^2(\Omega_i)}^2
	+ \sum_{d = k_i}^{d_i} \sum_{j \in I_i^d} 
	\| \Tr_j a_i \|_{L^2(\Omega_{i_j})}^2, \\
	\| \mathfrak{a} \|_{H \mathfrak{L}^k}^2 &:= 
	\| \mathfrak{a} \|_{L^2 \mathfrak{L}^k}^2
	+ \| \mathfrak{da} \|_{L^2 \mathfrak{L}^{k + 1}}^2.
\end{align*}
The inner products that naturally induce these norms are denoted by $(\cdot, \cdot)_{L^2 \mathfrak{L}^k}$ and $(\cdot, \cdot)_{H \mathfrak{L}^k}$, respectively.
The spaces $H \mathfrak{L}^k$ form a cochain complex which we refer to as the mixed-dimensional de Rham complex: 
\begin{equation} \label{eq: De Rham complex L}
	\begin{tikzcd}
	H \mathfrak{L}^0 \arrow[r,"\mathfrak{d}"]
	& H \mathfrak{L}^1 \arrow[r,"\mathfrak{d}"]
	& \dots \arrow[r,"\mathfrak{d}"]
	& H \mathfrak{L}^{n - 1} \arrow[r,"\mathfrak{d}"]
	& H \mathfrak{L}^n.
	\end{tikzcd}
\end{equation}
This complex has several key properties, which we present in the following Lemma. 
\begin{lemma}
	The complex \eqref{eq: De Rham complex L} satisfies the following:
	\begin{itemize}
		\item All exact forms are closed: each $\mathfrak{a} \in H \mathfrak{L}^k$ satisfies
		\begin{align} \label{eq: closedness}
			\mathfrak{d(da)} &= 0.
		\end{align}
		\item All closed forms are exact: for each $\mathfrak{a} \in H \mathfrak{L}^k$ with $\mathfrak{da} = 0$, there exists a $\mathfrak{b} \in H \mathfrak{L}^{k - 1}$ such that
		\begin{align} \label{eq: exactness}
			\mathfrak{db} &= \mathfrak{a}, &
			\| \mathfrak{b} \|_{H \mathfrak{L}^{k - 1}} &\lesssim \| \mathfrak{a} \|_{H \mathfrak{L}^k}
		\end{align}
	\end{itemize}
\end{lemma}
\begin{proof}
	The proof can be found in \cite{boon:md2018}.
\end{proof}

We represent this complex in terms of local spaces for each dimension. For $n = 3$, these local spaces are then organized in the following diagram:
\begin{equation} \label{eq: De Rham complex}
	\begin{tikzcd}
	H \mathfrak{L}^0 \arrow[d,"\mathfrak{d}"]  
	& \tx{H(\nabla, \Tr; \Omega^3)} \arrow[d,"\nabla"]\arrow[rd,"\Bbbd"] \\
	H \mathfrak{L}^1 \arrow[d,"\mathfrak{d}"] 
	& \tx{H(\nabla \times, \Tr; \Omega^3)} \arrow[d,"\nabla \times"]\arrow[rd,"\Bbbd"]
	& \tx{H(\nabla^\perp, \Tr; \Omega^2)} \arrow[d,"\nabla^\perp"]\arrow[rd,"\Bbbd"] \\
	H \mathfrak{L}^2 \arrow[d,"\mathfrak{d}"] 
	& \tx{H(\nabla\cdot, \Tr; \Omega^3)} \arrow[d,"\nabla \cdot"]\arrow[rd,"\Bbbd"] 
	& \tx{H(\nabla\cdot, \Tr; \Omega^2)} \arrow[d,"\nabla \cdot"]\arrow[rd,"\Bbbd"] 
	& \tx{H(\nabla\cdot, \Tr; \Omega^1)} \arrow[d,"\nabla \cdot"]\arrow[rd,"\Bbbd"] \\
	H \mathfrak{L}^3 
	& \tx{L^2(\Omega^3)}
	& \tx{L^2(\Omega^2)}
	& \tx{L^2(\Omega^1)}
	& \tx{L^2(\Omega^0)}
	\end{tikzcd}
\end{equation}
Here, $\nabla \cdot$ denotes the divergence tangential to each manifold $\Omega_i$, regardless of dimension. The curl is denoted by $\nabla \times$ in three dimensions and is given by the rotated gradient $\nabla^\perp$ on two-dimensional manifolds. The operator $\nabla$ at the top of the diagram represents the gradient on three-dimensional subdomains.

The local function spaces are given by subspaces of conventional Sobolev spaces with extra trace regularity. By defining $\nu$ as the outward, unit normal vector on $\partial \Omega_i$ for $i \in I$, we define
	\begin{align*}
		H(\nabla \cdot, \Tr; \Omega_i)
		&:= \{ a_i \in H(\nabla \cdot, \Omega_i)
		:\ 
		&(\nu \cdot a_i)|_{\partial_j \Omega_i} 
		&\in L^2(\Omega_{i_j})
		, \ &&\forall j \in I_i^{d_i - 1} \},
		\\
		H(\nabla \times, \Tr; \Omega_i)
		&:= \{ a_i \in H(\nabla \times, \Omega_i) 
		:\ 
		&(-\nu \times a_i)|_{\partial_j \Omega_i} 
		&\in H(\nabla \cdot, \Tr; \Omega_{i_j})
		, \ &&\forall j \in I_i^{d_i - 1} \},
		\\
		H(\nabla^\perp, \Tr; \Omega_i)
		&:= \{ a_i \in H(\nabla^\perp, \Omega_i) 
		:\ 
		&(\nu^\perp a_i)|_{\partial_j \Omega_i} 
		&\in H(\nabla \cdot, \Tr; \Omega_{i_j})
		, \ &&\forall j \in I_i^{d_i - 1} \},
		\\
		H(\nabla, \Tr; \Omega_i)
		&:= \{ a_i \in H(\nabla, \Omega_i)
		:\ 
		&(a_i)|_{\partial_j \Omega_i} 
		&\in H(\nabla \times, \Tr; \Omega_{i_j})
		, \ &&\forall j \in I_i^{d_i - 1} \}.
	\end{align*}
The spaces in diagram \eqref{eq: De Rham complex} are then defined as the product of these spaces over all $i \in I^d$ for a given dimension $d$. 

\section{Regular Decomposition}
\label{sec:regular_decomposition}

The aim of this section is to show that the conventional regular decomposition of differential $k$-forms can be generalized to the mixed-dimensional setting. For that purpose, we first recall the fixed-dimensional regular decomposition in the continuous case. Then, we introduce the subspace of $H \Lambda^k$, that contains functions with higher regularity, and the analogous subspace of $H \mathfrak{L}^k$. In turn, this gives the main ingredients in the derivation of the mixed-dimensional regular decomposition.

\subsection{Fixed-dimensional Regular Decomposition} \label{sub:fixeddim_regular_decomposition}
We start with presenting the regular decomposition in the context of the de Rham complex \eqref{eq: FD de Rham represent}. Given $\Omega_i$ with $d_i = 3$, we follow the results in \cite{hiptmair:xu} and provide the regular decomposition of $ H(\nabla \cdot , \Omega_i) $ and $ H(\nabla \times, \Omega_i) $ in the following theorems.
\begin{theorem}[Regular decomposition of $ H(\nabla \times, \Omega_i) $] \label{thm:regular_Hcurl}
	For any $ \bm{q} \in H(\nabla \times, \Omega_i) $, there exist functions $ \bm{a} \in (H(\nabla, \Omega_i))^3 $ and $ c \in H(\nabla, \Omega_i) $ such that
	\begin{subequations}
		\begin{align}
		\bm{q} & = \bm{a} + \nabla c, \\
		\| \bm{a} \|_{H(\nabla, \Omega_i)} + \| c \|_{H(\nabla, \Omega_i)} & \lesssim \| \bm{q} \|_{H(\nabla \times, \Omega_i)}.
		\end{align}
	\end{subequations}
\end{theorem}
\begin{theorem}[Regular decomposition of $ H(\nabla \cdot, \Omega_i) $] \label{thm:regular_Hdiv}
	For any $ \bm{q} \in H(\nabla \cdot, \Omega_i) $, there exist functions $ \bm{a}, \bm{c} \in (H(\nabla, \Omega_i))^3 $ such that
	\begin{subequations}
		\begin{align}
		\bm{q} & = \bm{a} + \nabla \times \bm{c}, \\
		\| \bm{a} \|_{H(\nabla, \Omega_i)} + \| \bm{c} \|_{H(\nabla, \Omega_i)} & \lesssim \| \bm{q} \|_{H(\nabla \cdot, \Omega_i)}.
		\end{align}
	\end{subequations}
\end{theorem}
Now, let $ H_0(\nabla \times, \Omega_i) $ and $ H_0(\nabla \cdot, \Omega_i) $ be the subspaces of $ H(\nabla \cdot , \Omega_i) $ and $ H(\nabla \times, \Omega_i) $, respectively, with zero trace on the boundary $ \partial \Omega_i $. Also, denote the vector function space $ \bm{H}_0(\nabla, \Omega_i) = \{ \bm{u} \in (H(\nabla, \Omega_i))^3, \bm{u}|_{\partial \Omega_i} = \bm{0} \} $. The "boundary aware" regular decompositions from \cite{hiptmair2019boundary} are given in the following theorems.

\begin{theorem}[Regular decomposition of $ H_0(\nabla \times, \Omega_i) $] \label{thm:regular_Hcurl_0}
	For any $ \bm{q} \in H_0(\nabla \times, \Omega_i) $, there exist functions $ \bm{a} \in \bm{H}_0(\nabla, \Omega_i) $ and $ c \in H_0(\nabla, \Omega_i) $ such that
	\begin{subequations}
		\begin{align}
		\bm{q} & = \bm{a} + \nabla c, \\
		\| \bm{a} \|_{\bm{H}_0(\nabla, \Omega_i)} + \| c \|_{H_0(\nabla, \Omega_i)} & \lesssim \| \bm{q} \|_{H_0(\nabla \times, \Omega_i)}.
		\end{align}
	\end{subequations}
\end{theorem}
\begin{theorem}[Regular decomposition of $ H_0(\nabla \cdot, \Omega_i) $] \label{thm:regular_Hdiv_0}
	For any $ \bm{q} \in H_0(\nabla \cdot, \Omega_i) $, there exist functions $ \bm{a}, \bm{c} \in \bm{H}_0(\nabla, \Omega_i) $ such that
	\begin{subequations}
		\begin{align}
		\bm{q} & = \bm{a} + \nabla \times \bm{c}, \\
		\| \bm{a} \|_{\bm{H}_0(\nabla, \Omega_i)} + \| \bm{c} \|_{\bm{H}_0(\nabla, \Omega_i)} & \lesssim \| \bm{q} \|_{H_0(\nabla \cdot, \Omega_i)}.
		\end{align}
	\end{subequations}
\end{theorem}
Our derivation of the mixed-dimensional regular decomposition relies on the fact that the regular decompositions in above theorems are possible on the individual sub-manifolds $ \Omega_i $ and then combined together, by taking special care of the traces.

\subsection{Mixed-dimensional Regular Decomposition} \label{sub:mixeddim_regular_decomposition}

For $i \in I$, let us first introduce the subspace of $k$-forms with increased regularity, denoted by $H^1 \Lambda^{k} (\Omega_i) \subseteq H \Lambda^{k} (\Omega_i)$ such that
\begin{align*}
	H^1 \Lambda^{k} (\Omega_i) \cong (H(\nabla, \Omega_i))^{C_{d_i, k}}.
\end{align*} 
Here the notation $\cong$ means that the spaces are isomorphic. The exponent is given by the binomial coefficient 
$C_{d_i, k} := \left( \begin{smallmatrix} d_i \\ k \end{smallmatrix} \right)$,
which is the dimension of the space of differential $k$-forms on a $d_i$-manifold (see e.g. \cite{spivak2018calculus}, Thm 4-5). We consider the space as zero if the exponent is zero, e.g. if $k > d_i$.

With the local spaces defined, let $H^1 \mathfrak{L}^k \subseteq H \mathfrak{L}^k$ be the space of regular mixed-dimensional $k$-forms, given by
\begin{align} \label{eq: regular k-forms}
	H^1 \mathfrak{L}^k
	&:= \prod_{i \in I} \{ a_i \in H^1 \Lambda^{k_i}(\Omega_i): 
	\Tr_j a_i \in H^1 \Lambda^{k_i}(\Omega_{i_j}), 
	\ \forall j \in I_i \},
\end{align}
and endowed with the norm
\begin{align} \label{eq: regular norm}
	\| \mathfrak{a} \|_{H^1 \mathfrak{L}^k} &:= 
	\sum_{i \in I} \| a_i \|_{H^1 \Lambda^{k_i}(\Omega_i)}
	+ \sum_{j \in I_i} 
	\| \Tr_j a_i \|_{H^1 \Lambda^{k_i}(\Omega_{i_j})}.
\end{align}

We note two properties of the space $H^1 \mathfrak{L}^k$. First, in the special case of $i \in I^{n - k}$, we have $k_i = 0$. Since $H^1 \Lambda^0(\Omega_i) = H \Lambda^0(\Omega_i)$, it follows that
\begin{subequations} \label{eq: properties H1}
\begin{align} \label{eq: H1 = H for k_i=0}
	\iota_i H^1 \mathfrak{L}^k &= \iota_i H \mathfrak{L}^k, 
	& \forall i \in I^{n - k}. 
\end{align}
Secondly, the jump operator preserves the increased regularity of $H^1 \mathfrak{L}^k$:
\begin{align} \label{eq: jump in H1}
	\Bbbd H^1 \mathfrak{L}^k \subseteq H^1 \mathfrak{L}^{k + 1}.
\end{align}
\end{subequations}


The following lemma provides the local regular decompositions on each $ \Omega_i $ by using the known results in fixed-dimensional setting in \Cref{sub:fixeddim_regular_decomposition}.

\begin{lemma} \label{lem:local regular decomposition}
	Given $q_i \in \iota_i H \mathfrak{L}^k$ with $i \in I$, then there exists a pair $(a_i, c_i) \in \iota_i H^1 \mathfrak{L}^k \times \iota_i H^1 \mathfrak{L}^{k - 1}$ such that
	\begin{align*}
	q_i &= a_i + \emph{d} c_i
	& \text{and}&
	& \| a_i \|_{H^1 \mathfrak{L}^k} 
	+ \| c_i \|_{H^1 \mathfrak{L}^{k - 1}} 
	&\lesssim \| q_i \|_{H \mathfrak{L}^k}.
	\end{align*}
\end{lemma}
\begin{proof}
	We consider the four possible cases for $n \le 3$. With reference to diagram \eqref{eq: De Rham complex}, these cases are represented by the main diagonal (Case \ref{local case 0}) and the off-diagonal components in the three bottom rows (Cases \ref{local case 1}--\ref{local case 3}). 
	\begin{enumerate}[label=Case \Alph*:, ref=\Alph*]
		\item \label{local case 0}
		$k = n - d_i$. We note that this means that $k_i = 0$ and $\iota_i H \mathfrak{L}^k = \iota_i H^1 \mathfrak{L}^k$ by \eqref{eq: H1 = H for k_i=0}. Setting $a_i := q_i$ yields the result.
		\item \label{local case 1}
		$k = n, \ d_i > 0$. In this case, we have $q_i \in L^2(\Omega_i)$. We introduce $a_i \in H_0^1(\Omega_i)$ as the solution to the following minimization problem,
		\begin{subequations} \label{eq: step 1}
			\begin{align}
			\min_{a_i \in H_0^1(\Omega_i)} & \tfrac12 \| a_i \|_{1, \Omega_i}^2
			& \text{ subject to }
			\Pi_{\mathbb{R}, i} a_i &= \Pi_{\mathbb{R}, i} q_i,
			\end{align}
			with $\Pi_{\mathbb{R}, i}$ denoting the $L^2$-projection onto constants on $\Omega_i$.
			Secondly, we define the bounded $c_i \in (H_0^1(\Omega_i))^{d_i}$ such that
			\begin{align}
			\nabla \cdot c_i &= (I - \Pi_{\mathbb{R}, i}) (q_i - a_i).
			\end{align}
		\end{subequations}
		Since the divergence represents the exterior derivative $\myd$ in this case, it follows that $q_i = a_i + \myd c_i$ with $a_i \in \iota_i H^1 \mathfrak{L}^k$ and $c_i \in \iota_i H^1 \mathfrak{L}^{k - 1}$. Importantly, $a_i$ and $c_i$ have zero trace on $\partial \Omega_i$. This property will be advantageous in the remaining cases.
		
		\item \label{local case 2}
		$k = n - 1, \ d_i > 1$. In this case, we require $L^2$ regularity of traces on manifolds of codimension one. For $j \in I_i^{d_i - 1}$, let us denote
		\begin{align*}
		q_j := \Tr_j q_i.
		\end{align*}
		By definition of $H \mathfrak{L}^{n - 1}$, we have that $q_j \in L^2(\Omega_{i_j}) = L^2(\partial_j \Omega_i)$. Since $\partial_j \Omega_i$ is a manifold of dimension $d_i - 1$, we can use Case \ref{local case 1} to find $a_j \in H_0^1(\partial_j \Omega_i)$ and $c_j \in (H_0^1(\partial_j \Omega_i))^{d_i - 1}$ such that
		\begin{align*}
		q_j = a_j + \myd c_j.
		\end{align*}
		Note that both $a_j$ and $c_j$ have zero trace on the boundary of $\partial_j \Omega_i$. Hence, all $a_j$ (respectively $c_j$) can be combined to form a function in $H^1(\partial \Omega_i)$ (respectively $(H^1(\partial \Omega_i))^{d_i - 1}$). These boundary functions are extended harmonically into $\Omega_i$ to form $a_i^* \in \iota_i H^1 \mathfrak{L}^k$ and $c_i^* \in \iota_i H^1 \mathfrak{L}^{k - 1}$ such that
		\begin{align*}
		\Tr_j a_i^* &= a_j, &
		\Tr_j c_i^* &= c_j, &
		\forall j \in I_i^{d_i - 1}
		\end{align*}
		The regularity of these extensions in the domain $\Omega_i$ is a result of the fact that all $a_j$ and $c_j$ are zero at tips and reentrant corners.
		
		Next, we note that $q_i - (a_i^* + \myd c_i^*)$ has zero trace on $\partial_j \Omega_i$ for all $j \in I_i^{d_i - 1}$. Hence, we apply a regular decomposition respecting homogeneous boundary conditions to obtain
		\begin{align*}
		q_i - (a_i^* + \myd c_i^*)
		= a_i^0 + \myd c_i^0,
		\end{align*}
		such that $a_i^0 \in H^1 \Lambda^{k_i}(\Omega_i)$ and $c_i^0 \in H^1 \Lambda^{k_i - 1}(\Omega_i)$ have zero trace on each $\partial_j \Omega_i$ with $j \in I_i^{d_i - 1}$.
		It follows that $a_i^0 \in \iota_i H^1 \mathfrak{L}^k$ and $c_i^0 \in \iota_i H^1 \mathfrak{L}^{k - 1}$. We conclude the construction by setting $a_i := a_i^* + a_i^0$ and $c_i := c_i^* + c_i^0$.
		
		\item \label{local case 3}
		$k = n - 2, \ d_i > 2$. The only case not covered so far is $d_i = n = 3$. Following the same arguments as above, we first denote $q_j := \Tr_j q_i$ and then use the construction from Case \ref{local case 2} to obtain 
		\begin{align*}
		q_j &= a_j + \myd c_j, & j \in I_i^2.
		\end{align*}
		Next, without loss of generality, we let $j_1, j_2 \in I_i^2$ and consider the index $j_{1,2} \in I_i^1$ such that $\partial_{j_{1,2}} \Omega_i$ forms a one-dimensional interface between $\partial_{j_1} \Omega_i$ and $\partial_{j_2} \Omega_i$. It follows that 
		\begin{align*}
		\Tr_{j_{1,2}} a_{j_1}
		&= \Tr_{j_{1,2}} a_{j_2} &
		\Tr_{j_{1,2}} c_{j_1}
		&= \Tr_{j_{1,2}} c_{j_2}
		\end{align*}
		since both traces are equal to the unique constructions on $\partial_{j_{1,2}} \Omega_i$ from Case \ref{local case 1}, in particular \eqref{eq: step 1}. This means that, when there are more interfaces, by combining all $a_j$ with $j \in I_i^2$, a function is formed in $(H^1(\partial \Omega_i))^2$. Analogously, the combination of all $c_j$ with $j \in I_i^2$ forms a function in $H^1(\partial \Omega_i )$. 
		
		The construction is finalized in the same way as in Case \ref{local case 2}. In short, we first introduce a harmonic extension of the boundary functions to form $a_i^* \in \iota_i H^1 \mathfrak{L}^k$ and $c_i^* \in \iota_i H^1 \mathfrak{L}^{k - 1}$. Then, $a_i^0 \in \iota_i H^1 \mathfrak{L}^k$ and $c_i^0 \in \iota_i H^1 \mathfrak{L}^{k - 1}$ are constructed using a regular decomposition of $q_i - (a_i^* + \myd c_i^*)$ respecting homogeneous boundary conditions. Finally, we set $a_i := a_i^* + a_i^0$ and $c_i := c_i^* + c_i^0$.
	\end{enumerate}
\end{proof}

\begin{remark}
	In the previous lemma, we have frequently used the fact that operators $ \Tr $ and $ \operatorname{d} $ commute. The derivation of this property for functions on each $ \Omega_i $ is straightforward so, to shorten the presentation, we leave it out and use the commutative property in the proofs when needed.
\end{remark}

Now we are ready to present the main result of this section, namely the mixed-dimensional regular decomposition, in the following theorem. 

\begin{theorem}[Mixed-dimensional Regular Decomposition]
\label{thm: regular decomposition}
	Given $\mathfrak{q} \in H \mathfrak{L}^k$, then there exists a pair $(\mathfrak{a}, \mathfrak{c}) \in H^1 \mathfrak{L}^k \times H^1 \mathfrak{L}^{k - 1}$ such that
	\begin{align*}
		\mathfrak{q} &= \mathfrak{a + d c}
		& \text{and}&
		& \| \mathfrak{a} \|_{H^1 \mathfrak{L}^k} + \| \mathfrak{c} \|_{H^1 \mathfrak{L}^{k - 1}} &\lesssim \| \mathfrak{q} \|_{H \mathfrak{L}^k }
	\end{align*}
\end{theorem}
\begin{proof} 
	Given $k$, we construct $\mathfrak{a} = (a_i)_{i \in I}$ and $\mathfrak{c} = (c_i)_{i \in I}$ by marching through the corresponding row in diagram \eqref{eq: De Rham complex}. We initialize both functions as $ \mathfrak{a} = 0 $  and $ \mathfrak{c} = 0 $, and redefine each component according to the following four, sequential steps.
	\begin{enumerate}
		\item
		If $k > 0$, consider $i \in I^n$. Lemma~\ref{lem:local regular decomposition} gives us $a_i \in \iota_i H^1 \mathfrak{L}^k$ and $c_i \in \iota_i H^1 \mathfrak{L}^{k - 1}$ such that
		\begin{subequations} \label{eq: proof q = a + dc}
		\begin{align}
			 q_i &= a_i + \myd c_i.
		\end{align}
		We repeat the above for all $i \in I^n$ and continue with step 2.

		\item
		If $k > 1$, consider $i \in I^{n - 1}$. We use Lemma~\ref{lem:local regular decomposition} to define $\tilde{a}_i \in \iota_i H^1 \mathfrak{L}^k $ and $c_i \in \iota_i H^1 \mathfrak{L}^{k - 1}$ such that
		\begin{align*}
			 q_i &= \tilde{a}_i + \myd c_i.
		\end{align*}
		Noting that $(q_i - \myd c_i) \in \iota_i H^1 \mathfrak{L}^k$ and using \eqref{eq: jump in H1}, we set
		\begin{align}
			a_i &:= q_i - \myd c_i - \iota_i \Bbbd \mathfrak{c}.  
		\end{align}
		where $ \mathfrak{c} $ has non-zero components $ c_j, \, j \in I^n $ defined in step 1. We repeat this construction for all $i \in I^{n - 1}$ and continue with step 3.

		\item 
		If $k > 2$, repeat step 2 with $i \in I^{n - 2}$ and continue with step 4.
		
		\item
		In any case, consider $i \in I^{n - k}$. We have $k_i = 0$ and note that $\iota_i H \mathfrak{L}^k = \iota_i H^1 \mathfrak{L}^k$ from \eqref{eq: H1 = H for k_i=0}. Hence, we use \eqref{eq: jump in H1} to set
		\begin{align}
			a_i := q_i - \iota_i \Bbbd \mathfrak{c}.
		\end{align}
		where $ \mathfrak{c} $ has non-zero components defined in steps 1--3. This construction is repeated for all $i \in I^{n - k}$.
		\end{subequations}
	\end{enumerate}

	The four steps give us $\mathfrak{a} := (a_i)_{i \in I} \in H^1 \mathfrak{L}^k$ and $\mathfrak{c} := (c_i)_{i \in I} \in H^1 \mathfrak{L}^{k - 1}$ and we collect \eqref{eq: proof q = a + dc} to conclude
	\begin{align*}
		\mathfrak{q} = \mathfrak{a + dc}.
	\end{align*}
	The bound follows by the construction and Lemma~\ref{lem:local regular decomposition}.
\end{proof}

A byproduct of the mixed-dimensional regular decomposition is the so-called regular inverse of the mixed-dimensional differential $\mathfrak{d}$ as shown in the following corollary. 

\begin{corollary}[Mixed-dimensional Regular Inverse] \label{cor: regular inverse}
	Given $\mathfrak{q} \in H \mathfrak{L}^k$, then there exists $\mathfrak{a} \in H^1 \mathfrak{L}^k$ such that
	\begin{align*}
		\mathfrak{d (q - a)} &= 0 
		& \text{and}&
		& \| \mathfrak{a} \|_{H^1 \mathfrak{L}^k} &\lesssim \| \mathfrak{q} \|_{H \mathfrak{L}^k }
	\end{align*}
\end{corollary}
\begin{proof}
	Follows from Theorem~\ref{thm: regular decomposition} and the fact that $\mathfrak{ddc} = 0$ from \eqref{eq: closedness}.
\end{proof}

\section{Discrete Regular Decomposition}
\label{sec:discretization}
In this section, we introduce the discrete version of the regular decomposition (Theorem \ref{thm: regular decomposition}).  To this end, let $h$ be the typical mesh size and the subscript $h$ describe discrete entities. We first present the conventional discrete regular decomposition in a fixed-dimensional setting. Then, we introduce the structure-preserving discretization of the mixed-dimensional geometry $ \Omega $ and function spaces $ H \mathfrak{L}^k $. We finalize the section with deriving the discrete mixed-dimensional regular decomposition.

\subsection{Fixed-dimensional Discrete Regular Decomposition} \label{sub:fixeddim_discrete_regular_decomposition}
Let $ H_h(\nabla, \Omega_i) $, $H_h(\nabla \times, \Omega_i)$, and $H_h(\nabla \cdot, \Omega_i)$ denote the conforming finite element approximations of the functions spaces $ H(\nabla, \Omega_i) $, $ H(\nabla \times, \Omega_i) $ and $ H(\nabla \cdot, \Omega_i) $. In addition, let $ \Pi^{\nabla \times}_h :  H(\nabla \times, \Omega_i) \to  H_h(\nabla \times, \Omega_i)$ and $ \Pi^{\nabla \cdot}_h : H(\nabla \cdot, \Omega_i) \to H_h(\nabla \cdot, \Omega_i) $ be the stable projection operators. In connection to \Cref{sub:fixeddim_regular_decomposition}, the discrete analogues of \Cref{thm:regular_Hcurl} and \Cref{thm:regular_Hdiv} are given below.

	\begin{theorem}[Regular decomposition of $ H_h(\nabla \times, \Omega_i) $] \label{thm:discrete_regular_Hcurl}
		For any $ \bm{q}_h \in H_h(\nabla \times, \Omega_i) $, there exist vector functions $ \bm{a}_h \in (H_h(\nabla, \Omega_i))^3 $, $ \bm{b}_h \in H_h(\nabla \times, \Omega_i) $ and a scalar function $ c_h \in H_h(\nabla, \Omega_i) $ such that
		\begin{subequations}
			\begin{align}
			\bm{q}_h & = \Pi^{\nabla \times}_h \bm{a}_h + \bm{b}_h + \nabla c_h, \\
			\| \Pi^{\nabla \times}_h \bm{a}_h \|_{H(\nabla \times, \Omega_i)} + \| h^{-1} \bm{b}_h \|_{L^2(\Omega_i)} + \| c_h \|_{H(\nabla, \Omega_i)} & \lesssim \| \bm{q}_h \|_{H(\nabla \times, \Omega_i)}. \label{eq:discrete_regular_Hcurl_bound}
			\end{align}
		\end{subequations}
	\end{theorem}
	\begin{theorem}[Regular decomposition of $ H_h(\nabla \cdot, \Omega_i) $] \label{thm:discrete_regular_Hdiv}
		For any $ \bm{q}_h \in H_h(\nabla \cdot, \Omega_i) $, there exist vector functions $ \bm{a}_h, \bm{e}_h \in (H_h(\nabla, \Omega_i))^3 $, $ \bm{b}_h \in H_h(\nabla \cdot, \Omega_i)  $ and $ \bm{f}_h \in H_h(\nabla \times, \Omega_i) $ such that
		\begin{subequations}
			\begin{align}
			\bm{q}_h = \Pi^{\nabla \cdot}_h \bm{a}_h + \bm{b}_h + \nabla \times (\Pi^{\nabla \times}_h \bm{e}_h & + \bm{f}_h), \\
			\| \Pi^{\nabla \cdot}_h \bm{a}_h \|_{H(\nabla \cdot, \Omega_i)} + \| h^{-1} \bm{b}_h \|_{L^2(\Omega_i)} + \| \bm{e}_h \|_{H(\nabla \times, \Omega_i)} + \| h^{-1} \bm{f}_h \|_{L^2(\Omega_i)} & \lesssim \| \bm{q}_h \|_{H(\nabla \cdot, \Omega_i)}. \label{eq:discrete_regular_Hdiv_bound}
			\end{align}
		\end{subequations}
	\end{theorem}
These discrete regular decompositions reveal the structure that we aim to preserve in the mixed-dimensional setting. Specifically, the stability of the decompositions in the sense of bounds \eqref{eq:discrete_regular_Hcurl_bound} and \eqref{eq:discrete_regular_Hdiv_bound} will in turn provide us with robust preconditioners by the theory of the auxiliary space methods. We give a short overview of the auxiliary space preconditioning theory later in \Cref{sub:auxiliary_space} and focus first on the derivation of the mixed-dimensional analogue to \Cref{thm:discrete_regular_Hcurl} and \Cref{thm:discrete_regular_Hdiv}.

\subsection{Mixed-dimensional Discretization} \label{sub:mixeddim_discretization}

First, we introduce a shape-regular simplicial partition of $\Omega$, denoted by $\Omega_h = \bigcup_{i \in I} \Omega_{i, h}$. The grid is constructed such that it conforms to all lower-dimensional manifolds and all grids are matching. In order to preserve the regular decomposition on the discrete level, structure preserving discretization in the mixed-dimensional setting should be considered.  Let us introduce $H_h \mathfrak{L}^k$ as the discretization of $H \mathfrak{L}^k$ defined on $\Omega_h$. Using the notation of finite element exterior calculus \cite{arnold:falk:winther_FEEC}, we consider the family of reduced finite elements (i.e. elements of the first kind) and set
\begin{align} \label{eq: choice of spaces}
 	H_h \mathfrak{L}^k = \prod_{i \in I} P_r^- \Lambda^{k_i}(\Omega_{i, h}).
\end{align}

The lowest-order case ($r = 1$) in the three-dimensional setting ($n = 3$) gives us the following, discrete de Rham complex:
\begin{equation} \label{eq: discrete de rham}
	\begin{tikzcd}
	H_h \mathfrak{L}^0 \arrow[d,"\mathfrak{d}"] 
	& \mathbb{P}_1(\Omega_h^3) \arrow[d,"\nabla"]\arrow[rd,"\Bbbd"] \\
	H_h \mathfrak{L}^1 \arrow[d,"\mathfrak{d}"] 
	& \mathbb{N}_0^e(\Omega_h^3) \arrow[d,"\nabla \times"]\arrow[rd,"\Bbbd"]
	& \mathbb{P}_1(\Omega_h^2) \arrow[d,"\nabla^\perp"]\arrow[rd,"\Bbbd"] \\
	H_h \mathfrak{L}^2 \arrow[d,"\mathfrak{d}"] 
	& \mathbb{N}_0^f(\Omega_h^3) \arrow[d,"\nabla\cdot"]\arrow[rd,"\Bbbd"] 
	& \mathbb{RT}_0(\Omega_h^2) \arrow[d,"\nabla\cdot"]\arrow[rd,"\Bbbd"] 
	& \mathbb{P}_1(\Omega_h^1) \arrow[d,"\nabla \cdot"]\arrow[rd,"\Bbbd"] \\
	H_h \mathfrak{L}^3 
	& \mathbb{P}_0(\Omega_h^3)
	& \mathbb{P}_0(\Omega_h^2)
	& \mathbb{P}_0(\Omega_h^1)
	& \mathbb{P}_0(\Omega_h^0)
	\end{tikzcd}
\end{equation}
Here, $\mathbb{P}_1$, $\mathbb{RT}_0$, and $\mathbb{P}_0$ denote linear Lagrange, lowest-order Raviart-Thomas, and piecewise constant finite element spaces, respectively. $\mathbb{N}_0^e$ and $\mathbb{N}_0^f$ represent the edge-based and face-based N\'ed\'elec elements of lowest order, respectively.

We introduce the stable projection operators $\Pi_h^k: H \mathfrak{L}^k \mapsto H_h \mathfrak{L}^k$ such that 
\begin{align}
\| (I - \Pi_h^k) \mathfrak{a} \|_{L^2 \mathfrak{L}^k} & \lesssim h \| \mathfrak{a} \|_{H^1 \mathfrak{L}^k}, \quad \forall \, \mathfrak{a} \in H^1\mathfrak{L}^k, \label{ine:Pi-approx}
\end{align}
and the following diagram commutes,
\begin{equation} \label{eq: commuting diagram}
	\begin{tikzcd}
	H \mathfrak{L}^0 \arrow[r,"\mathfrak{d}"] \arrow[d,"\Pi_h^0"]
	& H \mathfrak{L}^1 \arrow[r,"\mathfrak{d}"] \arrow[d,"\Pi_h^1"]
	& H \mathfrak{L}^2 \arrow[r,"\mathfrak{d}"] \arrow[d,"\Pi_h^2"]
	& H \mathfrak{L}^3 \arrow[d,"\Pi_h^3"]
	\\
	H_h \mathfrak{L}^0 \arrow[r,"\mathfrak{d}"] 
	& H_h \mathfrak{L}^1 \arrow[r,"\mathfrak{d}"] 
	& H_h \mathfrak{L}^2 \arrow[r,"\mathfrak{d}"] 
	& H_h \mathfrak{L}^3 	
	\end{tikzcd}
\end{equation}
In the mixed-dimensional setting, such a bounded projection can be constructed by combining fixed-dimensional cochain projections on each $\Omega_i$ together.  For the construction of fixed-dimensional cochain projections, we refer to~\cite{falk2014local}.

\begin{lemma}[Exactness] \label{lem: exactness discrete}
	Given the commuting projection operators $\Pi_h^k$ exist, then all discrete closed forms are exact. Thus, for each $\mathfrak{q}_h \in H_h \mathfrak{L}^k$ with $\mathfrak{dq}_h = 0$, there exists a $\mathfrak{c}_h \in H_h \mathfrak{L}^{k - 1}$ such that
	\begin{align*}
		\mathfrak{dc}_h = \mathfrak{q}_h.
	\end{align*}
\end{lemma}
\begin{proof}
	Assume $\mathfrak{q}_h \in H_h \mathfrak{L}^k$ with $\mathfrak{dq}_h = 0$ given. Since $H_h \mathfrak{L}^k \subset H \mathfrak{L}^k$, we use the exactness of the mixed-dimensional De Rham complex \eqref{eq: exactness} to find $\mathfrak{c} \in H \mathfrak{L}^{k - 1}$ such that
	\begin{align*}
		\mathfrak{dc} = \mathfrak{q}_h.
	\end{align*}
	Setting $\mathfrak{c}_h := \Pi_h^{k - 1} \mathfrak{c}$, we obtain
	\begin{align*}
		\mathfrak{dc}_h
		= \mathfrak{d} \Pi_h^{k - 1} \mathfrak{c}
		= \Pi_h^k \mathfrak{d} \mathfrak{c}
		= \Pi_h^k \mathfrak{q}_h
		= \mathfrak{q}_h,
	\end{align*}
	which completes the proof. 
\end{proof}

\subsection{Mixed-dimensional Discrete Regular Decomposition}
\label{sub:mixeddim_discrete_regular_decomposition}

This section is devoted to deriving the discrete regular decomposition, i.e. the discrete analogue to Theorem~\ref{thm: regular decomposition}.  We first require the following preparatory lemma.
\begin{lemma} \label{lem: semidiscrete decomposition}
	Given $\mathfrak{q}_h \in H_h \mathfrak{L}^k$, then there exists a semi-discrete pair $(\mathfrak{a}, \mathfrak{f}_h) \in H^1 \mathfrak{L}^k \times H_h \mathfrak{L}^{k - 1}$ such that
	\begin{align*}
		\mathfrak{q}_h &= \Pi_h^k \mathfrak{a} + \mathfrak{d} \mathfrak{f}_h
		& \text{and}&
		& \| \mathfrak{a} \|_{H^1 \mathfrak{L}^k} + \| \mathfrak{f}_h \|_{H \mathfrak{L}^{k-1}} &\lesssim \| \mathfrak{q}_h \|_{H \mathfrak{L}^k }.
	\end{align*}
\end{lemma}
\begin{proof}
	Since $\mathfrak{q}_h \in H_h \mathfrak{L}^k \subset H \mathfrak{L}^k$, we use Corollary~\ref{cor: regular inverse} to construct $\mathfrak{a} \in H^1 \mathfrak{L}^k$ such that
	\begin{align*}
		\mathfrak{d} (\mathfrak{q}_h - \mathfrak{a}).
		= 0
	\end{align*}
	Next, we use $\mathfrak{d} H_h \mathfrak{L}^k \subseteq H_h \mathfrak{L}^{k + 1}$ and the commutativity of the projection operators to derive
	\begin{align*}
		\mathfrak{d} \mathfrak{q}_h
		 = \Pi_h^{k + 1} \mathfrak{d} \mathfrak{q}_h
		 = \Pi_h^{k + 1} \mathfrak{d} \mathfrak{a}
		 = \mathfrak{d} \Pi_h^k \mathfrak{a}.
	\end{align*}
	We thus have $\mathfrak{d} (\mathfrak{q}_h - \Pi_h^k \mathfrak{a}) = 0$, i.e. $\mathfrak{q}_h - \Pi_h^k \mathfrak{a}$ is a closed form in $H_h \mathfrak{L}^k$. From Lemma~\ref{lem: exactness discrete}, a $\mathfrak{f}_h \in H_h \mathfrak{L}^{k - 1}$ exists such that
	\begin{align*}
		\mathfrak{d} \mathfrak{f}_h = \mathfrak{q}_h - \Pi_h^k \mathfrak{a}.
	\end{align*}
	By Corollary~\ref{cor: regular inverse}, we have $\| \mathfrak{a} \|_{H^1 \mathfrak{L}^k} \lesssim \| \mathfrak{q}_h \|_{H\mathfrak{L}^k} $.  In addition, by~\eqref{ine:Pi-approx}, we have
	\begin{align*}
	\| \mathfrak{d}\mathfrak{f}_h \|_{L^2\mathfrak{L}^k}& \lesssim \| \mathfrak{q}_h \|_{L^2\mathfrak{L}^k} + \| \Pi_h^k \mathfrak{a}\|_{L^2\mathfrak{L}^k} \leq  \| \mathfrak{q}_h \|_{L^2\mathfrak{L}^k} + \| \mathfrak{a}\|_{L^2\mathfrak{L}^k} + \| (I - \Pi_h^k) \mathfrak{a} \|_{L^2\mathfrak{L}^k} \\
	& \lesssim \| \mathfrak{q}_h \|_{L^2\mathfrak{L}^k} + \| \mathfrak{a} \|_{H^1 \mathfrak{L}^k} \lesssim \| \mathfrak{q}_h \|_{H\mathfrak{L}^k}.
	\end{align*} 
	Since the choice of $\mathfrak{f}_h$ is not unique, we choose a special $\mathfrak{f}_h$ such that the following Poincar\'{e} inequality holds
	\begin{equation*}
	\| \mathfrak{f}_h \|_{H \mathfrak{L}^{k-1}} \lesssim \| \mathfrak{d} \mathfrak{f}_h \|_{L^2 \mathfrak{L}^k}.
	\end{equation*}
	Then we have $\| \mathfrak{f}_h \|_{H \mathfrak{L}^{k-1}} \lesssim \| \mathfrak{q}_h \|_{H \mathfrak{L}^k}$, which completes the proof.  
\end{proof}

To further develop the decomposition in the discrete setting, let us introduce $H_h^1 \mathfrak{L}^k \subseteq H^1 \mathfrak{L}^k$ as the discretization of the regular $k$-forms from \eqref{eq: regular k-forms}. For the given choice of discrete spaces \eqref{eq: choice of spaces}, the regular spaces are given by
\begin{align*}
	H_h^1 \mathfrak{L}^k = \prod_{i \in I} (P_r^- \Lambda^0(\Omega_{i, h}))^{C_{d_i, k_i}}.
\end{align*}
Again, the exponent is given by
$C_{d_i, k_i} := \left( \begin{smallmatrix} d_i \\ k_i \end{smallmatrix} \right)$. 
In the lowest-order case with $r = 1$, this means that
\begin{align*}
	H_h^1 \mathfrak{L}^k = \prod_{i \in I} (\mathbb{P}_1(\Omega_{i, h}))^{C_{d_i, k_i}}.
\end{align*}
In other words, all discrete forms with increased regularity are given by (tuples of) nodal Lagrange elements. Similar to \eqref{eq: properties H1}, we note that
\begin{subequations}
\begin{align} \label{eq: H1 = H for k_i=0 discrete}
	\iota_i H_h^1 \mathfrak{L}^k &= \iota_i H_h \mathfrak{L}^k, 
	& i \in I^{n - k}, \\
	\Bbbd H_h^1 \mathfrak{L}^k &\subseteq H_h^1 \mathfrak{L}^{k + 1}.
\end{align}
\end{subequations}
Let $\mathcal{P}_h^k$ be the projection operator onto the discretized space of regular $k$-forms. $\mathcal{P}_h^k: H^1 \mathfrak{L}^k \mapsto H_h^1 \mathfrak{L}^k$ is stable and has the following property:
\begin{align} \label{eq: approximation}
	\| \mathcal{P}_h^k \mathfrak{a} \|_{H^1\mathfrak{L}^k} \lesssim \| \mathfrak{a} \|_{H^1 \mathfrak{L}^k} &, &
	\| (I - \mathcal{P}_h^k) \mathfrak{a} \|_{L^2 \mathfrak{L}^k}
	&\lesssim h \| \mathfrak{a} \|_{H^1 \mathfrak{L}^k},
\end{align}
for all $\mathfrak{a} \in H^1 \mathfrak{L}^k$.  
Now we are ready to present the mixed-dimensional discrete regular decomposition as follows.

\begin{theorem}[Mixed-dimensional Discrete Regular Decomposition] \label{thm: discrete triplet}
	Given $\mathfrak{q}_h \in H_h \mathfrak{L}^k$, then there exists a pair $(\mathfrak{a}_h, \mathfrak{f}_h) \in H_h^1 \mathfrak{L}^k \times H_h \mathfrak{L}^{k - 1}$ and a high-frequency term $\mathfrak{b}_h \in H_h \mathfrak{L}^k $ such that
	\begin{align*}
		\mathfrak{q}_h &= \Pi_h^k \mathfrak{a}_h + \mathfrak{b}_h + \mathfrak{d} \mathfrak{f}_h
		& \text{and}&
		& \| \mathfrak{a}_h \|_{H^1 \mathfrak{L}^k} 
		+ \| h^{-1} \mathfrak{b}_h \|_{L^2 \mathfrak{L}^k}
		+ \| \mathfrak{f}_h \|_{H \mathfrak{L}^{k - 1}} 
		&\lesssim \| \mathfrak{q}_h \|_{H \mathfrak{L}^k }.
	\end{align*}
\end{theorem}
\begin{proof}
	Given the decomposition from Lemma~\ref{lem: semidiscrete decomposition}, we further decompose $\mathfrak{a}$ using the projection operator $\mathcal{P}_h^k$ from \eqref{eq: approximation}:
	\begin{align*}
		\mathfrak{q}_h = \Pi_h^k \mathcal{P}_h^k \mathfrak{a} + \Pi_h^k (I - \mathcal{P}_h^k) \mathfrak{a} + \mathfrak{d} \mathfrak{f}_h.
	\end{align*}
	By defining $\mathfrak{a}_h := \mathcal{P}_h^k \mathfrak{a}$ and $\mathfrak{b}_h := \Pi_h^k (I - \mathcal{P}_h^k) \mathfrak{a}$, we obtain the desired format. To prove the boundedness, we use the stability of $\mathcal{P}_h^k$ and approximation properties~\eqref{ine:Pi-approx} and~\eqref{eq: approximation} to derive
		\begin{align*}
			\| \mathfrak{a}_h \|_{H^1 \mathfrak{L}^k} 
			&= \| \mathcal{P}_h^k \mathfrak{a} \|_{H^1 \mathfrak{L}^k}
			\lesssim \| \mathfrak{a} \|_{H^1 \mathfrak{L}^k} \lesssim \| \mathfrak{q}_h \|_{H\mathfrak{L}^k} , \\
			\| h^{-1} \mathfrak{b}_h \|_{L^2 \mathfrak{L}^k}
			&= \| h^{-1} \Pi_h^k (I - \mathcal{P}_h^k) \mathfrak{a} \|_{L^2 \mathfrak{L}^k}
			\lesssim \| h^{-1} (I - \mathcal{P}_h^k) \mathfrak{a} \|_{L^2 \mathfrak{L}^k} + \| h^{-1} (I - \Pi_h^k) (I- \mathcal{P}_h^k) \mathfrak{a}\|_{L^2 \mathfrak{L}^k} \\
			&  
			\lesssim \| \mathfrak{a} \|_{H^1 \mathfrak{L}^k} +  \| (I - \mathcal{P}_h^k) \mathfrak{a} \|_{H^1 \mathfrak{L}^k} \lesssim \| \mathfrak{a} \|_{H^1 \mathfrak{L}^k}.
		\end{align*}
	Combining these estimates with the bound of Lemma~\ref{lem: semidiscrete decomposition} proves the result.
\end{proof}

In practice, it is useful to integrate discrete regular decompositions so that all the components have improved regularity except the high-frequency parts.  Such a result is shown as follows.  

\begin{corollary}[Integrated Mixed-dimensional Discrete Regular Decomposition] \label{coro: discrete decomposition}
	Given $\mathfrak{q}_h \in H_h \mathfrak{L}^k$, then there exist a regular pair $(\mathfrak{a}_h, \mathfrak{c}_h) \in H_h^1 \mathfrak{L}^k \times H_h^1 \mathfrak{L}^{k - 1}$ and a high-frequency pair $(\mathfrak{b}_h, \mathfrak{e}_h) \in H_h \mathfrak{L}^k \times H_h \mathfrak{L}^{k - 1}$ such that
\begin{subequations}
	\begin{align}
		\Pi_h^k \mathfrak{a}_h + \mathfrak{b}_h + \mathfrak{d} ( \Pi_h^{k - 1} \mathfrak{c}_h + \mathfrak{e}_h) 
		&= 
		\mathfrak{q}_h \\
		\| \mathfrak{a}_h \|_{H^1 \mathfrak{L}^k} 
		+ \| h^{-1} \mathfrak{b}_h \|_{L^2 \mathfrak{L}^k}
		+ \| \mathfrak{c}_h \|_{H^1 \mathfrak{L}^{k - 1}} 
		+ \| h^{-1} \mathfrak{e}_h \|_{L^2 \mathfrak{L}^{k - 1}}
		&\lesssim \| \mathfrak{q}_h \|_{H \mathfrak{L}^k }.
	\end{align}
\end{subequations}
\end{corollary}
\begin{proof}
	Since $H_h^1 \mathfrak{L}^0 = H_h \mathfrak{L}^0$ by \eqref{eq: H1 = H for k_i=0 discrete}, the result is trivial for $k = 0$. We use the same argument in combination with Lemma~\ref{thm: discrete triplet} to conclude that the case $k = 1$ follows with $\mathfrak{e}_h = 0$. 

	We continue with $k > 1$.
	Using the decomposition from Lemma~\ref{thm: discrete triplet}, we obtain $(\mathfrak{a}_h, \mathfrak{b}_h, \mathfrak{f}_h) \in H_h^1 \mathfrak{L}^k \times H_h \mathfrak{L}^k \times H_h \mathfrak{L}^{k - 1}$ such that
	\begin{align*}
		\mathfrak{q}_h &= \Pi_h^k \mathfrak{a}_h + \mathfrak{b}_h + \mathfrak{d} \mathfrak{f}_h,
	\end{align*}
	with the associated bound. Applying Lemma~\ref{thm: discrete triplet} once more on $\mathfrak{f}_h \in H_h \mathfrak{L}^{k - 1}$, we have $(\mathfrak{c}_h, \mathfrak{e}_h, \mathfrak{g}_h) \in H_h^1 \mathfrak{L}^{k - 1} \times H_h \mathfrak{L}^{k - 1} \times H_h \mathfrak{L}^{k - 2}$ such that
	\begin{align*}
		\mathfrak{f}_h &= \Pi_h^{k - 1} \mathfrak{c}_h + \mathfrak{e}_h + \mathfrak{d} \mathfrak{g}_h
	\end{align*}
	Due to \eqref{eq: closedness}, we have $\mathfrak{ddg}_h = 0$ and the result follows.
\end{proof}

Similarly, in the discrete case, we also have the regular inverse as the byproduct of the discrete regular decomposition, which is stated in the following corollary. 

\begin{corollary}[Discrete Regular Inverse] \label{cor: discrete regular inverse}
	Given $\mathfrak{q}_h \in H_h \mathfrak{L}^k$, then there exist $\mathfrak{a}_h \in H_h^1 \mathfrak{L}^k$ and a high-frequency term $\mathfrak{b}_h \in H_h \mathfrak{L}^k$ such that
	\begin{align}
		\mathfrak{d} (\mathfrak{q}_h - \Pi_h^k \mathfrak{a}_h - \mathfrak{b}_h) &= 0 
		& \text{and}&
		& \| \mathfrak{a}_h \|_{H^1 \mathfrak{L}^k} 
		+ \| h^{-1} \mathfrak{b}_h \|_{L^2 \mathfrak{L}^k} 
		&\lesssim \| \mathfrak{q}_h \|_{H \mathfrak{L}^k }
	\end{align}
\end{corollary}
\begin{proof}
	Follows from Theorem~\ref{coro: discrete decomposition} and the fact that $\mathfrak{dd} = 0$ from \eqref{eq: closedness}.
\end{proof}

\section{Mixed-dimensional Auxiliary Space Preconditioner}
\label{sec:preconditioner}

Based on the (discrete) regular decomposition, we develop robust preconditioners for solving the following abstract mixed-dimensional problem: Find~$\mathfrak{q}_h \in H_h \mathfrak{L}^k$, such that
\begin{align} \label{eq:md_problem}
	(\mathfrak{q}_h, \tilde{\mathfrak{q}}_h)_{L^2 \mathfrak{L}^k} 
	+ (\mathfrak{d} \mathfrak{q}_h, \mathfrak{d} \tilde{\mathfrak{q}}_h)_{L^2 \mathfrak{L}^{k + 1}} 
	& = (\mathfrak{f},  \tilde{\mathfrak{q}}_h)_{L^2 \mathfrak{L}^k} 
	& \forall \, \tilde{\mathfrak{q}}_h \in H_h \mathfrak{L}^k.
\end{align}
This can be written into a linear system $ \mathfrak{A}^k \mathfrak{q}_h = \mathfrak{f} $, where $ \mathfrak{A}^k =\mathfrak{I} + \mathfrak{d}^* \mathfrak{d} $ is a symmetric positive definite operator on $ H_h \mathfrak{L}^k $, $ \mathfrak{I} $ is the identity mapping and $ \mathfrak{d}^* $ is the adjoint of $ \mathfrak{d} $. Our goal is to derive a preconditioner $ \mathfrak{B} $ for the problem \eqref{eq:md_problem} based on the fictitious or auxiliary space preconditioning theory developed in \cite{nepomnyaschikh1992decomposition, xu1996auxiliary, hiptmair:xu}.

\subsection{Abstract Theory of Auxiliary Space Preconditioning}
\label{sub:auxiliary_space}

We recall the framework of the auxiliary space theory. Assume $ V $ is a separable Hilbert space with an inner product $ a(\cdot, \cdot) $. We aim to find $ u \in V $ that solves
\begin{align} \label{eq:general_elliptic_problem}
	a(u, v) & = (f, v) & \forall v & \in V,  
\end{align}
or equivalently 
\begin{align} \label{eq:Au=f}
	A u & = f, & 
\end{align}
where $ A : V \mapsto V' $ is symmetric positive definite such that $ \langle A u, v \rangle = a(u, v) $. Here $V'$ is the dual of $V$.  Using $A$, the norm induced by $ a(\cdot, \cdot) $ can be denoted by $\| \cdot \|_A$ and we also consider another inner product $ s(\cdot, \cdot) $ on $ V $, which induces another norm $\| \cdot \|_S$ with $S$ being symmetric positive definite. 

For designing auxiliary preconditioners, let $ W_\ell $, $ \ell = 1, 2, \dots, L $, be auxiliary spaces with inner products $ a_\ell(\cdot, \cdot) $ that induces norms $\| \cdot \|_{A_\ell}$, where $ A_\ell : W_\ell \mapsto W_\ell' $ are linear operators defined as $ \langle A_\ell u_\ell, v_\ell \rangle = a_\ell(u_\ell, v_\ell) $, for $ u_\ell, v_\ell \in W_\ell $, $\ell = 1, 2, \dots, L$.  In addition, we assume that there are transfer operators $ \Pi_\ell : W_\ell \mapsto V $.   Finally, we define the auxiliary product space $ \bar{V} = V \times W_1 \times W_2 \times \dots \times W_L $, and then represent the inner product on $ \bar{V} $ as
\begin{align*}
	\bar{a}(\bar{v}, \bar{v}) & = s(v, v) + \sum\limits_{\ell = 1}^{L} a_l(w_\ell, w_\ell) & \forall \, \bar{v} = (v, w_1, \dots, w_L) & \in \bar{V}.
\end{align*}
Using the fictitious or  auxiliary space method, the preconditioner $ B : V' \mapsto V $ for the linear problem \eqref{eq:Au=f} is defined as
\begin{equation} \label{eq:aux_space_precond}
	B = S^{-1} + \sum\limits_{\ell = 1}^{L} \Pi_\ell A_\ell^{-1} \Pi_\ell^*,
\end{equation}
where $S^{-1}$ is the so-called smoother operator. The following Lemma from \cite{hiptmair:xu}, which can be viewed as a special case of the fictitious lemma~\cite{nepomnyaschikh1992decomposition}, gives a bound on the condition number $ \kappa(BA) $.

\begin{lemma} \label{lem:aux_space_cond}
	Assume the following conditions hold:
	\begin{enumerate}
		\item There exist $ c_\ell > 0, \, \ell = 1, \dots, L $ such that $ \| \Pi_\ell w_\ell \|_A \leq c_\ell \| w_\ell \|_{A_\ell}, \quad \forall \, w_\ell \in W_\ell $.
		\item There exist $ c_s > 0 $ such that $ \| v \|_A \leq c_s \| v \|_S, \quad \forall \, v \in V $.
		\item For every $ v \in V $, there exists a decomposition $ v = v_0 + \sum\limits_{\ell = 1}^{L} w_\ell, \, v_0 \in V, \, w_\ell \in W_\ell $ and $ c_0 > 0 $ such that
		\begin{equation*}
			\| v_0 \|_S^2 + \sum\limits_{\ell = 1}^{L} \| w_\ell \|_{A_\ell}^2 \leq c_0 \| v \|_A^2.
		\end{equation*}
	\end{enumerate}
	Then $ \kappa(BA) \leq c_0^2 (c_s^2 + c_1^2 + \dots + c_L^2) $.
\end{lemma}

If all the bounds in \Cref{lem:aux_space_cond} are independent of discretization parameter $ h $ (and any other parameters), then $ B $ is a robust preconditioner for $ A $ in \eqref{eq:Au=f}. The auxiliary space preconditioner $B$ \eqref{eq:aux_space_precond} can be viewed as additive version.  As mentioned in~\cite{hiptmair:xu}, naturally, we can also apply auxiliary spaces successively to obtain a multiplicative auxiliary space preconditioner, which is also robust under the same conditions, we refer to~\cite{hu2013combined} for details.  Note that instead of directly applying operators $ A_\ell^{-1} $, we can replacing them by their spectrally equivalent approximations, $B_{\ell}$.  As long as the constants in the spectral equivalence are independent of physical and discretization parameters, the resulting auxiliary space preconditioners remain robust.   

\begin{remark}
	For example, \cite{hiptmair:xu} shows that the fixed-dimensional discrete regular decomposition in \Cref{thm:discrete_regular_Hdiv} follows the conditions of \Cref{lem:aux_space_cond}. Let $ V = H_h(\nabla \cdot, \Omega_i) $ and for $ u, v \in V $ let the bilinear form $ a(u, v) = (u, v) + (\nabla \cdot u, \nabla \cdot v) $ in \eqref{eq:general_elliptic_problem}. The auxiliary space theory gives the following preconditioner for solving \eqref{eq:Au=f}. Take the auxiliary spaces $ W_1 = (H_h(\nabla, \Omega_i))^3 $, $ W_2 = H_h(\nabla \times, \Omega_i) $, $ W_3 = (H_h(\nabla, \Omega_i))^3 $ and transfer operators $ \Pi_1 = \Pi^{\nabla \cdot}_h $, $ \Pi_2 = \nabla \times $, $ \Pi_3 = \nabla \times \Pi^{\nabla \times}_h $. With certain choices of smoothers $ S^{\nabla \cdot} $ and $ S^{\nabla \times} $ on $ H_h(\nabla \cdot, \Omega_i) $ and $ H_h(\nabla \times, \Omega_i) $ (for example, Jacobi smoother), respectively, we get
	\begin{align}
	B = (S^{\nabla \cdot})^{-1} + \Pi^{\nabla \cdot}_h A_{reg}^{-1} (\Pi^{\nabla \cdot}_h)^{*} & + \nabla \times (S^{\nabla \times})^{-1} (\nabla \times)^{*}
	+ \nabla \times \Pi^{\nabla \times}_h A_{reg}^{-1} (\nabla \times \Pi^{\nabla \times}_h)^{*},
	\end{align}
	where $ A_{reg} $ is the linear operator induced by the inner product on $ (H_h(\nabla, \Omega_i))^3 $. In the following section, we show that the similar preconditioner is feasible in the mixed-dimensional setting using \Cref{thm: discrete triplet}. 
\end{remark}

\subsection{Mixed-dimensional Preconditioner}
\label{sub:md_preconditioner}

Let us apply the theory in \Cref{sub:auxiliary_space} on the problem \eqref{eq:md_problem} to develop the auxiliary space preconditioner in the mixed-dimensional setting. Following \Cref{thm: discrete triplet}, for any $\mathfrak{q}_h \in H_h \mathfrak{L}^k$, there is a pair $(\mathfrak{a}_h, \mathfrak{f}_h) \in H_h^1 \mathfrak{L}^k \times H_h \mathfrak{L}^{k - 1}$ and a high-frequency term $\mathfrak{b}_h \in H_h \mathfrak{L}^k $ that allows the following decomposition 
\begin{equation*}
	\mathfrak{q}_h = \mathfrak{b}_h + \Pi_h^k \mathfrak{a}_h + \mathfrak{d} \mathfrak{f}_h.
\end{equation*}
Now, besides the original space $ V = H_h \mathfrak{L}^k $, we have two auxiliary spaces $ W_1 = H_h^1 \mathfrak{L}^k $ and $ W_2 = H_h \mathfrak{L}^{k-1} $. Furthermore, we take the transfer operator $ \Pi_1 = \Pi_h^k $ restricted to $ H_h^1 \mathfrak{L}^k $, i.e. $ \Pi_h^k: H_h^1 \mathfrak{L}^k \mapsto H_h \mathfrak{L}^k $, and $ \Pi_2 = \mathfrak{d} $.  
We write $\mathfrak{A}^k_{reg}$ for the symmetric positive definite linear operator defined by the inner product on the space $H^1\mathfrak{L}^k$, which can be viewed as (vector) Laplacian operators in the mixed-dimensional setting.   

For the sake of simplicity, we consider the Jacobi smoother. For a function $ \mathfrak{q}_h \in H_h \mathfrak{L}^k $, we have $\mathfrak{q}_h = \sum_{\mathfrak{e}} \mathfrak{q}_h^{\mathfrak{e}} $, where $\mathfrak{q}_h^{\mathfrak{e}} \in \operatorname{span}\{\mathfrak{e}\}$ where $\mathfrak{e}$ denotes a degree of freedom defined on either a node, edge, face or cell of $ \Omega_h $.  Then the smoothing operator is characterized by the inner product
\begin{equation*}
s(\mathfrak{q}_h, \mathfrak{q}_h) = \sum_{\mathfrak{e}} \left( (\mathfrak{q}_h^{\mathfrak{e}} ,\mathfrak{q}_h^{\mathfrak{e}} )_{L^2\mathfrak{L}^k} + (\mathfrak{d}\mathfrak{q}_h^{\mathfrak{e}} , \mathfrak{d} \mathfrak{q}_h^{\mathfrak{e}} )_{L^2\mathfrak{L}^{k+1}} \right).  
\end{equation*}
This leads to a smoother $\mathfrak{S}^k$, which, in matrix representation, coincides with the diagonal of $\mathfrak{A}^k$. 

The auxiliary space preconditioner $ \mathfrak{B} : (H_h \mathfrak{L}^k)' \mapsto H_h \mathfrak{L}^k $ for \eqref{eq:md_problem} takes the following form
\begin{equation} \label{eq:md_precond_1}
	\mathfrak{B}^k = (\mathfrak{S}^k)^{-1} + \Pi_h^k (\mathfrak{A}^k_{reg})^{-1} (\Pi_h^k)^* + \mathfrak{d} (\mathfrak{A}^{k-1})^{-1} \mathfrak{d}^*.
\end{equation}
Here, $^*$ denotes the adjoint with respect to the $L^2\mathfrak{L}^k$ inner product and is the standard matrix transpose in the matrix representation.  

In order to show the bound $ \kappa(\mathfrak{B}^k \mathfrak{A}^k) \lesssim 1 $, we need to verify the conditions in \Cref{lem:aux_space_cond} and the results are summarized in the following theorem. 

\begin{theorem} \label{thm:condition_number_Bk}
	Using $\mathfrak{B}^k$ from \eqref{eq:md_precond_1} as a preconditioner for solving the linear system \eqref{eq:md_problem} leads to a condition number $ \kappa(\mathfrak{B}^k \mathfrak{A}^k) \lesssim 1 $, where the hidden constant depends only on $\Omega$ and shape regularity of the mesh.
\end{theorem}
\begin{proof}
	We verify the three conditions of \Cref{lem:aux_space_cond}:
\begin{enumerate}
	\item It follows from the properties of $\Pi_h^k$ and $\mathfrak{d}$ that 
	\begin{align}
		&\| \Pi_h^k \mathfrak{a}_h \|^2_{H \mathfrak{L}^k} = \| \Pi_h^k \mathfrak{a}_h \|^2_{L^2 \mathfrak{L}^k} + \| \mathfrak{d} \Pi_h^k \mathfrak{a}_h \|^2_{L^2\mathfrak{L}^{k+1}} 
		\lesssim
	 \| \mathfrak{a}_h \|^2_{H^1 \mathfrak{L}^k} + \| \mathfrak{d} \mathfrak{a}_h \|^2_{L^2\mathfrak{L}^{k+1}}  \lesssim \| \mathfrak{a}_h \|^2_{H^1 \mathfrak{L}^k}, \label{eq:aux_cond_1} \\
	&	\| \mathfrak{d} \mathfrak{f}_h \|_{H \mathfrak{L}^k} = \| \mathfrak{d} \mathfrak{f}_h \|_{L^2\mathfrak{L}^k} \leq \| \mathfrak{f}_h \|_{ H \mathfrak{L}^{k - 1} }, \label{eq:aux_cond_2}
	\end{align}	
	for any $ \mathfrak{a}_h \in H_h^1 \mathfrak{L}^k $ and $ \mathfrak{f}_h \in H_h \mathfrak{L}^{k - 1}$.  This verifies the first condition of \Cref{lem:aux_space_cond}.
	\item 
	Since each element has a finite number of neighbors, there is a small constant $ c_s > 0 $ such that
	\begin{align} \label{eq:aux_cond_3}
		(\mathfrak{q}_h, \mathfrak{q}_h)_{\mathfrak{A}^k} 
		&= \| \textstyle{ \sum_{\mathfrak{e}}} q^{\mathfrak{e}}_h  \|_{L^2 \mathfrak{L}^k}^2
		+ \| \textstyle{ \sum_{\mathfrak{e}}} \mathfrak{d}q^{\mathfrak{e}}_h \|_{L^2 \mathfrak{L}^{k + 1}}^2
		\leq c^2_s \sum\limits_{\mathfrak{e}} \left( \| q^{\mathfrak{e}}_h \|_{L^2 \mathfrak{L}^k}^2
		+ \| \mathfrak{d}q^{\mathfrak{e}}_h \|_{L^2 \mathfrak{L}^{k + 1}}^2 \right)
		= c_s^2 s(\mathfrak{q}_h, \mathfrak{q}_h),
	\end{align}
	which verifies the second condition of \Cref{lem:aux_space_cond}.
	\item Lastly, we can see from \Cref{thm: discrete triplet} that we only need to show the bound on $ \| \mathfrak{b}_h \|_{S} $,
	\begin{align} \label{eq:aux_cond_4}
		\| \mathfrak{b}_h \|_{S}^2 
		= \sum\limits_{\mathfrak{e}} \| b^{\mathfrak{e}}_h \|_{H \mathfrak{L}^k}^2
    	&= \sum\limits_{\mathfrak{e}} \left( \| b^{\mathfrak{e}}_h \|_{L^2 \mathfrak{L}^k}^2
		+ \| \mathfrak{d}b^{\mathfrak{e}}_h \|_{L^2 \mathfrak{L}^{k + 1}}^2 \right) \nonumber\\
	\text{(Inverse inequality)}		& \lesssim \sum\limits_{\mathfrak{e}} \left( \| b^{\mathfrak{e}}_h \|_{L^2 \mathfrak{L}^k}^2 
		+ \| h^{-1} b^{\mathfrak{e}}_h \|_{L^2 \mathfrak{L}^k}^2 \right) \nonumber\\
	\text{($L^2$-stability of the bases)}	& \lesssim \| \mathfrak{b}_h \|_{L^2 \mathfrak{L}^k}^2 
		+ \| h^{-1} \mathfrak{b}_h \|_{L^2 \mathfrak{L}^k}^2 \nonumber \\
	\text{(\Cref{thm: discrete triplet})}	& \lesssim \| \mathfrak{q}_h \|^2_{H\mathfrak{L}^k}.
	\end{align}

\end{enumerate} 
Therefore, by applying \Cref{lem:aux_space_cond}, we have that~$ \kappa(\mathfrak{B}^k \mathfrak{A}^k) \lesssim 1 $.
\end{proof}

We note that it is possible to choose different smoother, such as Gauss-Seidel smoother. In fact, one could use any $s(\cdot, \cdot)$ that is spectral equivalent to $\| h^{-1} \cdot \|^2_{L^2\mathfrak{L}^k} + \| \cdot \|^2_{L^2\mathfrak{L}^k}$.

Finally, we can integrate the regular decomposition into the preconditioner by utilizing \Cref{coro: discrete decomposition} to further expand $ \mathfrak{B}^k $. This will be especially useful in \Cref{sec:practical example} when designing a preconditioner for a parameter-dependent saddle point problem in practice. Similarly as before, we set $ V = H_h \mathfrak{L}^k $, $ W_1 = H_h^1 \mathfrak{L}^k $, $ W_2 = H_h \mathfrak{L}^{k-1} $ and $ W_3 = H_h^1 \mathfrak{L}^{k-1} $. The transfer operators are then $ \Pi_1 = \Pi_h^k $ restricted to $ H_h^1 \mathfrak{L}^k $, $ \Pi_2 = \mathfrak{d} $ and $ \Pi_3 = \mathfrak{d} \Pi_h^{k-1} $ restricted to $ H_h^1 \mathfrak{L}^{k-1} $. Again, we still use Jacobi smoother for the sake of simplicity here. The preconditioner $ \mathfrak{B} $ now has the following form

\begin{equation} \label{eq:md_precond_2}
	\mathfrak{B}^k = (\mathfrak{S}^k)^{-1} + \Pi_h^k (\mathfrak{A}^k_{reg})^{-1} (\Pi_h^k)^* + \mathfrak{d} (\mathfrak{S}^{k-1})^{-1} \mathfrak{d}^* + \mathfrak{d} \Pi_h^{k-1} (\mathfrak{A}^{k-1}_{reg})^{-1} (\Pi_h^{k-1})^* \mathfrak{d}^*.
\end{equation}
\vskip 5mm
\noindent The next corollary shows the bound $\kappa(\mathfrak{B}^k\mathfrak{A}^k) \lesssim 1$ using $\mathfrak{B}^k$ from~\eqref{eq:md_precond_2}.

\begin{corollary}
	Using $\mathfrak{B}^k$ from \eqref{eq:md_precond_2} as a preconditioner for solving the linear system \eqref{eq:md_problem} leads to a condition number $ \kappa(\mathfrak{B}^k \mathfrak{A}^k) \lesssim 1 $, where the hidden constant depends only on $\Omega$ and shape regularity of the mesh.
\end{corollary}
\begin{proof}
Since this preconditioner results from the one in \eqref{eq:md_precond_1} with further decomposing functions in $ H_h \mathfrak{L}^{k-1} $, the conditions in \Cref{lem:aux_space_cond} follow from \eqref{eq:aux_cond_1}--\eqref{eq:aux_cond_4}, which gives the desired result. 
\end{proof}


\begin{remark}
We emphasize that instead of directly applying inverses of operators $ \mathfrak{A}^k_{reg} $ and $ \mathfrak{A}^{k-1}_{reg} $, we can replace them by spectrally equivalent operators, i.e. spectrally equivalent inner products on $ H_h^1 \mathfrak{L}^k $ and $ H_h^1 \mathfrak{L}^{k-1} $.  Possible choices are multigrid methods and domain decomposition methods. 
\end{remark}

\begin{remark} \label{rem:general_problem}
For the sake of simplicity, we use model problem~\eqref{eq:Au=f} to derive the mixed-dimensional auxiliary space preconditioner.  It is also applicable to the following general problem
\begin{align*} 
\tau (\mathfrak{q}_h, \tilde{\mathfrak{q}}_h)_{L^2 \mathfrak{L}^k} 
+ (\mathfrak{d} \mathfrak{q}_h, \mathfrak{d} \tilde{\mathfrak{q}}_h)_{L^2 \mathfrak{L}^{k + 1}} 
& = (\mathfrak{f},  \tilde{\mathfrak{q}}_h)_{L^2 \mathfrak{L}^k} 
& \forall \, \tilde{\mathfrak{q}}_h \in H_h \mathfrak{L}^k.
\end{align*}
with $\tau > 0$.  In fact, such problem appears in the example presented in \Cref{sec:practical example} when the mixed-dimensional permeability is a constant.
\end{remark}

\section{A Practical Example: Flow in Fractured Porous Media} 
\label{sec:practical example}

This section presents a practical example in which the theory from the previous sections comes to use. We consider the setting of flow in fractured porous media in which fractures and intersections are modeled as lower-dimensional manifolds. The goal is to solve for a mass-conservative flow field consisting of a flux and a pressure variable. The flux $\mathfrak{q}$ is considered as a mixed-dimensional $(n-1)$-form whereas the pressure distribution $\mathfrak{p}$ is represented by a mixed-dimensional $n$-form \cite{boon:flow2018,nordbotten2017modeling}. With respect to the diagram \eqref{eq: De Rham complex}, this model therefore focuses on the bottom two rows.

We consider the natural case of $n = 3$. Then, the flux is defined as a 3-vector in the three-dimensional surroundings, a 2-vector in the two-dimensional fractures, and a scalar in the one-dimensional intersections between fractures. On the other hand, the pressure is defined as a scalar on all manifolds $\Omega_i$ with $i \in I$. 
In this case, in stead of using $\mathfrak{d}$, we denote $\mathfrak{D}$ as the mixed-dimensional differential, which is an analogue of the operator $\nabla$. Then we represent the complex \eqref{eq: De Rham complex L} in the same manner as \eqref{eq: FD de Rham represent}
\begin{equation*}
	\begin{tikzcd}
		H(\mathfrak{D}, \Omega) \arrow[r,"\mathfrak{D}"]
		& H(\mathfrak{D} \times, \Omega) \arrow[r,"\mathfrak{D} \times"] 
		& H(\mathfrak{D} \cdot, \Omega) \arrow[r,"\mathfrak{D} \cdot"] 
		& L^2(\Omega)
	\end{tikzcd}
\end{equation*}

The mixed formulation of a fracture flow problem governed by Darcy's law and conservation of mass is then given by: Find $(\mathfrak{q, p}) \in H(\mathfrak{D} \cdot, \Omega) \times L^2(\Omega) $ such that
\begin{subequations} \label{eq:flow_system}
\begin{align}
	(\mathfrak{K}^{-1} \mathfrak{q, \tilde{q}})_{L^2 \mathfrak{L}^{2}}
	- (\mathfrak{D \cdot \tilde{q}, p})_{L^2 \mathfrak{L}^3}
	&= 0, &
	\forall \, \tilde{\mathfrak{q}} &\in H(\mathfrak{D} \cdot, \Omega), \\
	(\mathfrak{D \cdot q, \tilde{p}})_{L^2 \mathfrak{L}^3}
	&= (\mathfrak{f, \tilde{p}})_{L^2 \mathfrak{L}^3}, &
	\forall \, \tilde{\mathfrak{p}} &\in L^2 \mathfrak{L}^3,
\end{align}
\end{subequations}
where
\begin{subequations} 
	\begin{align}
	(\mathfrak{K}^{-1} \mathfrak{q, \tilde{q}})_{L^2 \mathfrak{L}^{2}} 
	&:=
	\sum_{d = 1}^n \sum_{i \in I^d} (K^{-1} q_i, \tilde{q}_i )_{\Omega_i} 
	+ \sum_{j \in I_j^{d - 1}} (K_\nu^{-1} \nu_j \cdot q_i, \nu_j \cdot \tilde{q}_i )_{\partial_j \Omega_i} \\
	(\mathfrak{D \cdot q, \tilde{p}})_{L^2 \mathfrak{L}^3},
	&:=
	\sum_{d = 0}^n \sum_{i \in I^d} (\iota_i (\mathfrak{D \cdot q}), \tilde{\mathfrak{p}}_i)_{\Omega_i}.
	\end{align}
\end{subequations}
Here, $\mathfrak{f}$ is a given source term. $\mathfrak{K}$ is the mixed-dimensional permeability tensor given by a tangential and a normal component, denoted by $K$ and $K_\nu$, respectively. 



\subsection{Discrete Problem}

For the discretization, we follow \Cref{sec:discretization} and choose the finite element spaces given by the final two rows in diagram \eqref{eq: discrete de rham}. Note that this corresponds to the mixed finite element scheme presented and analyzed in \cite{boon:flow2018}. In short, we choose $H_h \mathfrak{L}^{2} \times H_h \mathfrak{L}^3$ as in diagram \eqref{eq: discrete de rham} and consider the discrete problem: Find $(\mathfrak{q}_h, \mathfrak{p}_h) \in H_h \mathfrak{L}^{2} \times H_h \mathfrak{L}^3$ such that
\begin{subequations} \label{eq:flow_system_h}
\begin{align}
	(\mathfrak{K}^{-1} \mathfrak{q}_h, \tilde{\mathfrak{q}}_h)_{L^2 \mathfrak{L}^{2}}
	- (\mathfrak{D \cdot \tilde{q}}_h, \mathfrak{p}_h)_{L^2 \mathfrak{L}^3}
	&= 0, &
	\forall \, \tilde{\mathfrak{q}}_h &\in H_h \mathfrak{L}^{2}, \\
	(\mathfrak{D \cdot q}_h, \mathfrak{\tilde{p}}_h)_{L^2 \mathfrak{L}^3}
	&= (\mathfrak{f, \tilde{p}}_h)_{L^2 \mathfrak{L}^3}, &
	\forall \, \tilde{\mathfrak{p}}_h &\in H_h \mathfrak{L}^3.
\end{align}
\end{subequations}
We briefly verify that problem \eqref{eq:flow_system_h} is well-posed. For that, we define the weighted norms
\begin{subequations} \label{eq:weighted_norms}
	\begin{align}
		\| \mathfrak{q}_h \|_{H_{\alpha}\mathfrak{L}^{2}}^2 &:= \| \mathfrak{K}^{-\frac{1}{2}} \mathfrak{q}_h \|_{L^2 \mathfrak{L}^{2}}^2 + \alpha \| \mathfrak{D} \cdot \mathfrak{q}_h \|_{L^2 \mathfrak{L}^3}^2, \\
		\| \mathfrak{p}_h \|_{H_{\alpha}\mathfrak{L}^3}^2 &:= \alpha^{-1} \| \mathfrak{p}_h \|_{L^2 \mathfrak{L}^3}^2.
	\end{align}
\end{subequations}
Here, the scalar $\alpha$ is chosen such that $\alpha \geq \mathfrak{K}_{\min}^{-1} $ with $ \mathfrak{K}_{\min} > 0 $ being the minimal eigenvalue of $ \mathfrak{K} $. In turn, we have
\begin{align} \label{eq: bound alpha}
	\| \mathfrak{K}^{-\frac{1}{2}} \mathfrak{q}_h \|_{L^2 \mathfrak{L}^{2}}^2 &\le
	\alpha \| \mathfrak{q}_h \|_{L^2 \mathfrak{L}^{2}}^2, &
	\forall \, \mathfrak{q}_h \in H_h \mathfrak{L}^{2}.
\end{align}
A key result in the analysis of this problem is that the pair of finite element spaces $ H_h \mathfrak{L}^{2} \times H_h \mathfrak{L}^3 $ satisfies the following inf-sup condition with respect to the weighted norms \eqref{eq:weighted_norms}.

\begin{lemma} \label{lem:inf_sup_B}
	There exists a constant $ \gamma_B > 0 $ independent of the discretization parameter $ h $ and the physical parameter $ \mathfrak{K} $ such that
	\begin{equation} \label{eq:inf_sup}
		\inf\limits_{\mathfrak{p}_h \in H_h \mathfrak{L}^3} \sup\limits_{\mathfrak{q}_h \in H_h \mathfrak{L}^{2}} \dfrac{-(\mathfrak{D} \cdot \mathfrak{q}_h, \mathfrak{p}_h)_{L^2 \mathfrak{L}^3}}{\| \mathfrak{q}_h \|_{H_{\alpha}\mathfrak{L}^{2}} \| \mathfrak{p}_h \|_{H_{\alpha}\mathfrak{L}^3}} \geq \gamma_B.
	\end{equation}
\end{lemma}
\begin{proof}
	For any given $ \mathfrak{p}_h \in H_h \mathfrak{L}^3 $, according to the inf-sup condition proven in \cite{boon:flow2018} (Lemma 3.2), there exists a $\mathfrak{q}_h \in H_h \mathfrak{L}^{2}$ such that
	\begin{align*}
		 -(\mathfrak{D} \cdot  \mathfrak{q}_h, \mathfrak{p_h})_{L^2 \mathfrak{L}^3} & = \| \mathfrak{p}_h \|_{L^2 \mathfrak{L}^3}^2, \\
		 \| \mathfrak{K}^{-\frac{1}{2}} \mathfrak{q}_h \|_{L^2 \mathfrak{L}^{2}}^2
		 + \| \mathfrak{D} \cdot \mathfrak{q}_h \|_{L^2 \mathfrak{L}^3}^2 
		 & \lesssim \| \mathfrak{p}_h \|_{L^2 \mathfrak{L}^3}^2.
	\end{align*}
	Using these properties and \eqref{eq: bound alpha}, it follows that
	\begin{align*}
		- (\mathfrak{D} \cdot  \mathfrak{q}_h, \mathfrak{p}_h)_{L^2 \mathfrak{L}^3} 
		& = \| \mathfrak{p}_h \|_{L^2 \mathfrak{L}^3}^2 \\
		& = \left(\alpha^{-\frac{1}{2}} \| \mathfrak{p}_h \|_{L^2 \mathfrak{L}^3} \right)	
		\left(\alpha^{\frac{1}{2}} \| \mathfrak{p}_h \|_{L^2 \mathfrak{L}^3} \right) \\
		& \gtrsim 
		\| \mathfrak{p}_h \|_{H_{\alpha}\mathfrak{L}^3} 
		\left(\alpha \| \mathfrak{q}_h \|_{L^2 \mathfrak{L}^{2}}^2 
		+ \alpha \| \mathfrak{D} \cdot \mathfrak{q}_h \|_{L^2 \mathfrak{L}^3}^2 \right)^\frac{1}{2} \\
		& \ge 
		\| \mathfrak{p}_h \|_{H_{\alpha}\mathfrak{L}^3} \| \mathfrak{q}_h \|_{H_{\alpha}\mathfrak{L}^{2}}.
	\end{align*} 
	This completes the proof. 
\end{proof}

Based on \Cref{lem:inf_sup_B}, we can show the well-posedness of problem \eqref{eq:flow_system_h} by introducing the following spaces and weighted norms. Let $ \mathfrak{X} := H_h \mathfrak{L}^{2} \times  H_h \mathfrak{L}^3$ and $ \mathfrak{X}' $ be the corresponding dual space. Let the energy norm on $ \mathfrak{X} $ be given by
\begin{equation} \label{eq:energy_norm}
	\vertiii{\mathfrak{x}}_{\mathfrak{X}}^2 = \vertiii{(\mathfrak{q}_h, \mathfrak{p}_h)}_{\mathfrak{X}}^2 = \| \mathfrak{q}_h \|_{H_{\alpha}\mathfrak{L}^{2}}^2 + \| \mathfrak{p}_h \|_{H_{\alpha}\mathfrak{L}^3}^2,
\end{equation}
which is induced by the inner product $ (\cdot, \cdot)_{\mathfrak{X}} $, i.e. $ (\mathfrak{x}, \mathfrak{x})_{\mathfrak{X}} =  \vertiii{\mathfrak{x}}_{\mathfrak{X}}^2 $. In addition, let us introduce the following composite bilinear form
\begin{equation} \label{eq:bilin_form_L}
	\mathcal{L}(\mathfrak{x}, \mathfrak{y}) 
	:= (\mathfrak{K}^{-1} \mathfrak{q}_h, \tilde{\mathfrak{q}}_h)_{L^2 \mathfrak{L}^{2}} 
	- (\mathfrak{D} \cdot \tilde{\mathfrak{q}}_h, \mathfrak{p}_h)_{L^2 \mathfrak{L}^3} 
	+ (\mathfrak{D} \cdot \mathfrak{q}_h, \tilde{\mathfrak{p}}_h)_{L^2 \mathfrak{L}^3},
\end{equation}
for $ \mathfrak{x} = (\mathfrak{q}_h, \mathfrak{p}_h) $ and $ \mathfrak{y} = (\tilde{\mathfrak{q}}_h, \tilde{\mathfrak{p}}_h) $. Now we can show the problem \eqref{eq:flow_system_h} is well-posed, as presented in the following theorem. 

\begin{theorem} \label{thm:well_posed_L}
	There exist constants $ \beta, \gamma > 0 $ independent of discretization parameter $ h $ and physical parameter $ \mathfrak{K} $ such that
	\begin{equation} \label{eq:inf_sup_L}
		\inf\limits_{\mathfrak{x} \in \mathfrak{X}} \sup\limits_{\mathfrak{y} \in \mathfrak{X}}
		\dfrac{ \mathcal{L}(\mathfrak{x}, \mathfrak{y}) }{\vertiii{ \mathfrak{x} }_{\mathfrak{X}} \vertiii{ \mathfrak{y} }_{\mathfrak{X}} } \geq \gamma
		\quad \text{and} \quad 
		| \mathcal{L}(\mathfrak{x}, \mathfrak{y}) | \leq \beta \vertiii{ \mathfrak{x} }_{\mathfrak{X}} \vertiii{ \mathfrak{y} }_{\mathfrak{X}}, \qquad \forall \, \mathfrak{x}, \mathfrak{y} \in \mathfrak{X}.
	\end{equation}
\end{theorem}
\begin{proof}
	Let $ \mathfrak{x} = (\mathfrak{q}_h, \mathfrak{p}_h) \in \mathfrak{X} $. Due to the inf-sup condition in \Cref{lem:inf_sup_B}, there exists $ \mathfrak{r}_h \in H_h \mathfrak{L}^{2} $ for this given $ \mathfrak{p}_h $ such that
	\begin{subequations} \label{eq:inf_sup_for_r}
		\begin{align} 
		- (\mathfrak{D} \cdot \mathfrak{r}_h, \mathfrak{p}_h)_{L^2 \mathfrak{L}^3} & \geq \gamma_B 
		\| \mathfrak{p}_h \|_{H_{\alpha} \mathfrak{L}^3}^2, \\
		\| \mathfrak{r}_h \|_{H_{\alpha}\mathfrak{L}^{2}}^2 & =
		\| \mathfrak{p}_h \|_{H_{\alpha} \mathfrak{L}^3}^2.
		\end{align}
	\end{subequations}
	Then, choose $ \mathfrak{y} = (\tilde{\mathfrak{q}}_h, \tilde{\mathfrak{p}}_h) $ such that $ \tilde{\mathfrak{q}}_h = \mathfrak{q}_h + \gamma_B \mathfrak{r}_h $ and $ \tilde{\mathfrak{p}}_h = \mathfrak{p}_h + \alpha \mathfrak{D} \cdot \mathfrak{q}_h $, and use \eqref{eq:inf_sup_for_r} together with Cauchy-Schwarz inequality, we have
	\begin{align*}
		\mathcal{L}(\mathfrak{x}, \mathfrak{y}) 
		& = (\mathfrak{K}^{-1} \mathfrak{q}_h, \mathfrak{q}_h + \gamma_B \mathfrak{r}_h)_{L^2 \mathfrak{L}^{2}} 
		- (\mathfrak{D} \cdot (\mathfrak{q}_h + \gamma_B \mathfrak{r}_h), \mathfrak{p}_h)_{L^2 \mathfrak{L}^3} 
		+ (\mathfrak{D} \cdot \mathfrak{q}_h, \mathfrak{p}_h + \alpha \mathfrak{D} \cdot \mathfrak{q}_h)_{L^2 \mathfrak{L}^3} \\
		& = \| \mathfrak{K}^{-\frac{1}{2}} \mathfrak{q}_h \|_{L^2 \mathfrak{L}^{2}}^2 
		+ \gamma_B (\mathfrak{K}^{-1} \mathfrak{q}_h, \mathfrak{r}_h)_{L^2 \mathfrak{L}^{2}} 
		-\gamma_B (\mathfrak{D} \cdot \mathfrak{r}_h, \mathfrak{p}_h)_{L^2 \mathfrak{L}^3} 
		+ \alpha \| \mathfrak{D} \cdot \mathfrak{q}_h \|_{L^2 \mathfrak{L}^3}^2 \\
		& \geq \| \mathfrak{K}^{-\frac{1}{2}} \mathfrak{q}_h \|_{L^2 \mathfrak{L}^{2}}^2 
		- \frac{1}{2} \| \mathfrak{K}^{-\frac{1}{2}} \mathfrak{q}_h \|^2_{L^2 \mathfrak{L}^{2}} - \frac{\gamma_B^2}{2} \| \mathfrak{K}^{-\frac{1}{2}} \mathfrak{r}_h \|^2_{L^2 \mathfrak{L}^{2}} 
		+ \gamma_B^2 \| \mathfrak{p}_h \|_{H_{\alpha} \mathfrak{L}^3}^2 
		+ \alpha \| \mathfrak{D} \cdot \mathfrak{q}_h \|_{L^2 \mathfrak{L}^3}^2 \\
		& \geq \frac{1}{2} \|\mathfrak{K}^{-\frac{1}{2}}  \mathfrak{q}_h \|_{L^2\mathfrak{L}^2}^2 
		- \frac{\gamma_B^2}{2} \| \mathfrak{r}_h \|^2_{H_{\alpha}\mathfrak{L}^2} 
		+ \gamma_B^2 \| \mathfrak{p}_h \|_{H_{\alpha} \mathfrak{L}^3}^2 
		+ \alpha \| \mathfrak{D} \cdot \mathfrak{q}_h \|_{L^2 \mathfrak{L}^3}^2\\
		& = \frac{1}{2} \|\mathfrak{K}^{-\frac{1}{2}}  \mathfrak{q}_h \|_{L^2\mathfrak{L}^2}^2  + \frac{\gamma_B^2}{2} \| \mathfrak{p}_h \|_{H_{\alpha} \mathfrak{L}^3}^2 
		+ \alpha \| \mathfrak{D} \cdot \mathfrak{q}_h \|_{L^2 \mathfrak{L}^3}^2\\
		& \geq \dfrac{1}{2} \min\{1, \gamma_B^2\} ( \| \mathfrak{q}_h \|_{H_{\alpha}\mathfrak{L}^{2}}^2 +  \| \mathfrak{p}_h \|_{H_{\alpha}\mathfrak{L}^3}^2 ) \\
		& = \dfrac{1}{2} \min\{1, \gamma_B^2\} \vertiii{ \mathfrak{x} }_{\mathfrak{X}}^2.
	\end{align*}
On the other hand, using continuity of the norms and Cauchy-Schwarz inequality, it is straighforward to verify that $ \vertiii{ \mathfrak{y} }_{\mathfrak{X}}^2 \leq \dfrac{\sqrt{2}}{2} \vertiii{ \mathfrak{x} }_{\mathfrak{X}}^2 $, and that gives the first condition in \eqref{eq:inf_sup_L}. The same arguments can be applied to get the second condition on $ \mathcal{L}(\cdot, \cdot) $ in \eqref{eq:inf_sup_L}, which concludes the proof.
\end{proof}

\subsection{Block Preconditioners based on Auxiliary Space Preconditioning}
\label{sub:preconditioner}


Let $ \langle \cdot, \cdot \rangle $ denote the duality pairing between a function space and its dual. The discrete system \eqref{eq:flow_system_h} can be represented by the following block operator form
\begin{equation} \label{eq:alg_form}
\mathcal{A}
\begin{pmatrix}
\mathfrak{q}_h \\ \mathfrak{p}_h
\end{pmatrix}
=
\begin{pmatrix}
0 \\ \mathfrak{f}
\end{pmatrix}
\quad
\text{ with }
\quad
\mathcal{A} :=
\begin{pmatrix}
A_{\mathfrak{q}} & -B^T \\
B & 0
\end{pmatrix},
\end{equation}
where 
$ \langle A_{\mathfrak{q}} \mathfrak{q}_h, \tilde{\mathfrak{q}}_h \rangle 
:= (\mathfrak{K}^{-1} \mathfrak{q}_h, \tilde{\mathfrak{q}}_h)_{L^2 \mathfrak{L}^{2}} $ and 
$\langle B \mathfrak{q}_h, \tilde{\mathfrak{p}}_h \rangle 
:= (\mathfrak{D} \cdot \mathfrak{q}_h, \tilde{\mathfrak{p}}_h)_{L^2 \mathfrak{L}^3} $.  

According to~\Cref{thm:well_posed_L}, $\mathcal{A}$ is an isomorphism with respect to the weighted energy norm~\eqref{eq:energy_norm}. Following the standard framework~\cite{mardal:winther}, the canonical block preconditioner for solving the linear system \eqref{eq:alg_form} is the Riesz operator $ \mathcal{B} : \mathfrak{X}' \mapsto \mathfrak{X} $ corresponding to the inner product $ (\cdot, \cdot)_\mathfrak{X} $, i.e.,
\begin{equation*} 
	(\mathcal{B} \mathfrak{f}, \mathfrak{x})_{\mathfrak{X}} =  \langle \mathfrak{f}, \mathfrak{x} \rangle, \qquad \forall \, \mathfrak{f} \in \mathfrak{X}', \; \mathfrak{x} \in \mathfrak{X}.
\end{equation*}
It follows from~\Cref{thm:well_posed_L} that,
\begin{equation}\label{eq:cond_no_B}
	\kappa(\mathcal{B} \mathcal{A}) = \| \mathcal{B} \mathcal{A} \|_{\mathscr{L}(\mathfrak{X}, \mathfrak{X})} \| (\mathcal{B} \mathcal{A})^{-1} \|_{\mathscr{L}(\mathfrak{X}, \mathfrak{X})} \leq \dfrac{\beta}{\gamma}.
\end{equation}
If $ \beta $ and $ \gamma $ are independent of the discretization and physical parameters, then $ \mathcal{B} $ is a robust preconditioner for linear system \eqref{eq:alg_form}. Based on the definition of the weighted energy norm~\eqref{eq:energy_norm}, the preconditioner $ \mathcal{B} $ takes the following block diagonal form
\begin{equation}
\label{eq:preconditioner}
	\mathcal{B} =
	\begin{pmatrix}
	A_{\mathfrak{q}} + \alpha B^T B & 0 \\
	0 & \alpha^{-1} A_{\mathfrak{p}}
	\end{pmatrix}^{-1}
	= 
	\begin{pmatrix}
	\left(A_{\mathfrak{q}} + \alpha B^T B \right)^{-1} & 0 \\
	0 & \alpha A_{\mathfrak{p}}^{-1}
	\end{pmatrix},
\end{equation}
where $ \langle A_{\mathfrak{p}} \mathfrak{p}_h, \mathfrak{p}_h \rangle := ( \mathfrak{p}_h, \mathfrak{p}_h)_{L^2 \mathfrak{L}^n}$.  
\begin{remark}
	The top block $ (A_{\mathfrak{q}} + \alpha B^T B)^{-1} $ in the preconditioner $ \mathcal{B} $ corresponds to applying the augmented Lagrangian method to a parameter-dependent problem \eqref{eq:alg_form}. The method is well-known and used in literature \cite{tuminaro:xu:zhu, lee:wu:xu:zikatanov, fortin2000augmented} for general elliptic problems since it effectively handles the difficulties in convergence of general iterative methods, such as the physical parameter $ \mathfrak{K} $ affecting the condition number of the linear system.
\end{remark}

In practice, directly inverting the diagonal blocks in~\eqref{eq:preconditioner} might not be feasible. To overcome this difficult, we replace the diagonal blocks by their spectrally equivalent approximation and propose the following block diagonal preconditioner, 
\begin{equation*}
\mathcal{M}_D = 
\begin{pmatrix}
M_{\mathfrak{q}} & 0 \\
0 & M_{\mathfrak{p}}
\end{pmatrix},
\end{equation*}
where
	\begin{alignat*}{3}
	c_{1, \mathfrak{q}} \langle M_{\mathfrak{q}} \mathfrak{q}_h, \mathfrak{q}_h \rangle \,
	& \leq  \langle (A_{\mathfrak{q}} + \alpha B^T B)^{-1} \mathfrak{q}_h , \mathfrak{q}_h \rangle \,
	& \leq c_{2, \mathfrak{q}} \langle M_{\mathfrak{q}} \mathfrak{q}_h, \mathfrak{q}_h \rangle,  \\ 
	c_{1, \mathfrak{p}} \langle M_{\mathfrak{p}} \mathfrak{p}_h, \mathfrak{p}_h \rangle \,
	& \leq  \hskip 0.8cm \langle \alpha A_{\mathfrak{p}}^{-1} \mathfrak{p}_h , \mathfrak{p}_h \rangle \,
	& \leq c_{2, \mathfrak{p}}  \langle M_{\mathfrak{p}} \mathfrak{p}_h, \mathfrak{p}_h \rangle, 
	\end{alignat*}
where $c_{1,\mathfrak{q}}$, $c_{1,\mathfrak{p}}$, $c_{2,\mathfrak{q}}$, and $c_{2,\mathfrak{p}}$ are positive constants independent of discretization and physical parameters. Following \cite{mardal:winther, budisa:hu} and using \Cref{thm:well_posed_L} and \eqref{eq:cond_no_B}, the condition number of $ \mathcal{M}_D \mathcal{A} $ can be directly estimated as 
\begin{equation*}
\kappa(\mathcal{M}_D \mathcal{A}) \leq \dfrac{\beta c_2}{ \gamma c_1}, 
\end{equation*}
for $c_2 = \max\{c_{2,\mathfrak{q}}, c_{2,\mathfrak{p}}\}$ and $c_1 = \min \{ c_{1,\mathfrak{q}}, c_{1,\mathfrak{p}} \}$. Again, if $\beta$, $\gamma$, $c_1$, and $c_2$ are independent of the discretization and physical parameters, then $\mathcal{M}_D$ is a robust preconditioner as well. 

Now we discuss our choices of $M_{\mathfrak{q}}$ and $M_{\mathfrak{p}}$. We start with $M_{\mathfrak{p}}$.  Due to the fact that the choice of finite element space for the pressure variable is piecewise constant, it follows that the corresponding mass matrix is diagonal and thus, easily invertible.  Therefore, we take $ M_{\mathfrak{p}} = \alpha A_{\mathfrak{p}}^{-1} $ and, naturally,  $ c_{1,\mathfrak{p}} = c_{2,\mathfrak{p}} = 1 $. 

Regarding $ M_{\mathfrak{q}} $, since the first block $ A_{\mathfrak{q}} + \alpha B^T B $ corresponds to the problem 
\begin{equation} \label{eq:flux-block}
(\mathfrak{K}^{-1} \mathfrak{q}_h, \mathfrak{\tilde{q}}_h)_{L^2\mathfrak{L}^2} + \alpha (\mathfrak{D}\cdot \mathfrak{q}_h, \mathfrak{D} \cdot \mathfrak{\tilde{q}}_h)_{L^2\mathfrak{L}^3}.
\end{equation}
It is quite challenging to solve it using traditional methods due to the large kernel of the operator $\mathfrak{D} \cdot$. 
Therefore, we propose to use the mixed-dimensional auxiliary space preconditioner~\eqref{eq:md_precond_2}, derived in \Cref{sub:md_preconditioner}. 
The form \eqref{eq:flux-block} can be viewed as a special case of the mixed dimensional problem~\eqref{eq:md_problem} when $n = 3$, $k = 2$, and certain coefficients are added.  Directly apply the auxiliary space preconditioner~\eqref{eq:md_precond_2}, we have
\begin{equation*} 
	\mathfrak{B}^2= 
	(\mathfrak{S}^2)^{-1} 
	+ \Pi_h^2 (\mathfrak{A}^2_{reg})^{-1} (\Pi_h^2)^* 
	+ (\mathfrak{D} \times) (\mathfrak{S}^1)^{-1} (\mathfrak{D} \times)^* 
	+ (\mathfrak{D} \times)( \Pi_h^1) (\mathfrak{A}^1_{reg})^{-1} (\Pi_h^1)^* (\mathfrak{D} \times)^* .
\end{equation*}
The smoothers $\mathfrak{S}^2$ and $\mathfrak{S}^1$ are chosen to satisfy the second condition in \Cref{lem:aux_space_cond}.  In our implementation, we use symmetric Gauss-Seidel smoothers for both cases.  Since the regular space $H_h^1 \mathfrak{L}^{2}$ is given by $d_i$-tuples of linear Lagrange elements on each $\Omega_{i, h}$ with $i \in I$ and $d_i \ge 1$ and $H_h^1 \mathfrak{L}^{1}$ is defined for $i \in I^3$ (respectively $I^2$) as a 3-vector field (respectively scalar field) of linear Lagrange elements on $\Omega_{i, h}$, $\mathfrak{A}_{reg}^{2}$ and $\mathfrak{A}_{reg}^{1}$ represent the (weighted) inner products on these spaces from \eqref{eq: regular norm}. Moreover, it is often advantageous to further substitute spectrally equivalent operators for $(\mathfrak{A}_{reg}^k)^{-1}$, $k=2,1$, denoted by $\mathfrak{B}^k_{reg}$, then the overall auxiliary space preconditioner for solving~\eqref{eq:flux-block} is

\begin{equation} \label{eq:md_precond_3}
\mathfrak{B}_{\mathfrak{q}} := 
(\mathfrak{S}^2)^{-1} 
+ \Pi_h^2 \mathfrak{B}^2_{reg}(\Pi_h^2)^* 
+ (\mathfrak{D} \times) (\mathfrak{S}^1)^{-1} (\mathfrak{D} \times)^* 
+ (\mathfrak{D} \times)( \Pi_h^1) \mathfrak{B}^1_{reg} (\Pi_h^1)^* (\mathfrak{D} \times)^*.
\end{equation}
\vskip 5mm
\noindent and our choice of $M_{\mathfrak{q}}$ is defined as solving~\eqref{eq:flux-block} by Generalize Minimal Residual (GMRES) method with $\mathfrak{B}_q$ as the preconditioner. In our implementation, $\mathfrak{B}^k_{ref}$, $k=2,1$, are defined by one W-cycle unsmoothed aggregation algebraic multigrid method. A theoretical study of their spectrally equivalence properties and thorough comparison of the different available choices is outside the scope of this work and are subjects of our future work. 

Lastly, we also consider two block triangular preconditioners
\begin{equation*}
	\mathcal{M}_L = 
	\begin{pmatrix}
		M_{\mathfrak{q}}^{-1} & 0 \\
		-B & M_{\mathfrak{p}}^{-1}
	\end{pmatrix}^{-1} 
	\text{ and }
	\mathcal{M}_U = 
	\begin{pmatrix}
		M_{\mathfrak{q}}^{-1} & B^T \\
		0 & M_{\mathfrak{p}}^{-1}
	\end{pmatrix}^{-1},
\end{equation*}
where $ \mathcal{M}_L $ serves as a uniform left and $ \mathcal{M}_U $ as a uniform right preconditioner for solving~\eqref{eq:alg_form}. It can be proven that $ \mathcal{M}_L $ and $ \mathcal{M}_U $ are so-called field-of-value (FoV) equivalent preconditioners based on the well-posdeness conditions \eqref{eq:inf_sup_L} and proper inner product induced by $\mathcal{M}_D$.  We refer the reader to \cite{loghin:wathen, budisa:hu, alder:hu:rodrigo:zikatanov, adler:gaspar:hu:ohm:rodrigo:zikatanov} for a more detailed theoretical analysis on these preconditioners and restrict our focus on their numerical performances in the next section. 


\section{Numerical Examples}
\label{sec:examples}

In this section, we propose several numerical tests to confirm the theory derived in previous sections. These tests are designed to emphasize common challenges related to mixed-dimensional problems, such as the geometric complexity and parameter heterogeneity. Also, the problems represent simplified mathematical models of common applications, in this case the model of flow in fractured porous media introduced in \Cref{sec:practical example}.

In each example, we generate separate simplicial grids on rock and fracture subdomains which combined produce mixed-dimensional geometry $ \Omega $. For the sake of simplicity, we assume that $ \Omega $ is of rectangular type and all the adjacent grids are matching. We want to point out that the analysis presented in this paper allows more flexibility in the geometrical structure. 

To solve the system \eqref{eq:alg_form}, we use a Flexible Generalized Minimal Residual (FGMRES) method as an \textit{outer} iterative solver and set the tolerance to be the relative residual less than $ 10^{-6} $. We precondition the outer FGMRES solver with the block preconditioners designed in \Cref{sub:preconditioner}, i.e., the block diagonal preconditioner $ \mathcal{M}_D $ and the block triangular preconditioners $ \mathcal{M}_L $ and $ \mathcal{M}_U $.  As mentioned, the pressure block $ \alpha^{-1} A_{\mathfrak{p}} $ is represented as a diagonal matrix using piecewise constant finite elements, thus the inverse is given straightforwardly. On the other hand, the flux block $ A_{\mathfrak{q}} + \alpha B^T B $ is approximated by $ M_{\mathfrak{q}} $ which is defined by GMRES method preconditioned by the mixed-dimensional auxiliary space preconditioner~$\mathfrak{B}_{\mathfrak{q}}$~\eqref{eq:md_precond_3}. We refer this as the inner solver with a relative residual tolerance set to $ 10^{-3}$.  To define $\mathfrak{B}_{\mathfrak{q}}$, we use symmetric Gauss-Seidel method as smoothers $ (\mathfrak{S}^2)^{-1} $ and $ (\mathfrak{S}^1)^{-1} $, and one application of W-cycle unsmoothed aggregation Algebraic Multigrid method (UA-AMG) as $ \mathfrak{B}^2_{reg} $ and $\mathfrak{B}^1_{reg} $.

For obtaining the mixed-dimensional geometry and discretization, we use the PorePy library \cite{porepy}, an open-source simulation tool for fractured and deformable porous media written in Python. The solving methods and preconditioners are implemented in HAZMATH library \cite{hazmath}, a finite element solver library written in C. The following numerical examples are performed on a workstation with an $8$-core 3GHz Intel Xeon "Sandy Bridge" CPU and 256 GB of RAM.

\subsection{Example: Three-dimensional Regular network}
\label{sub:example3}

\begin{figure}[htbp]
	\centering
	\includegraphics[width=0.45\textwidth]{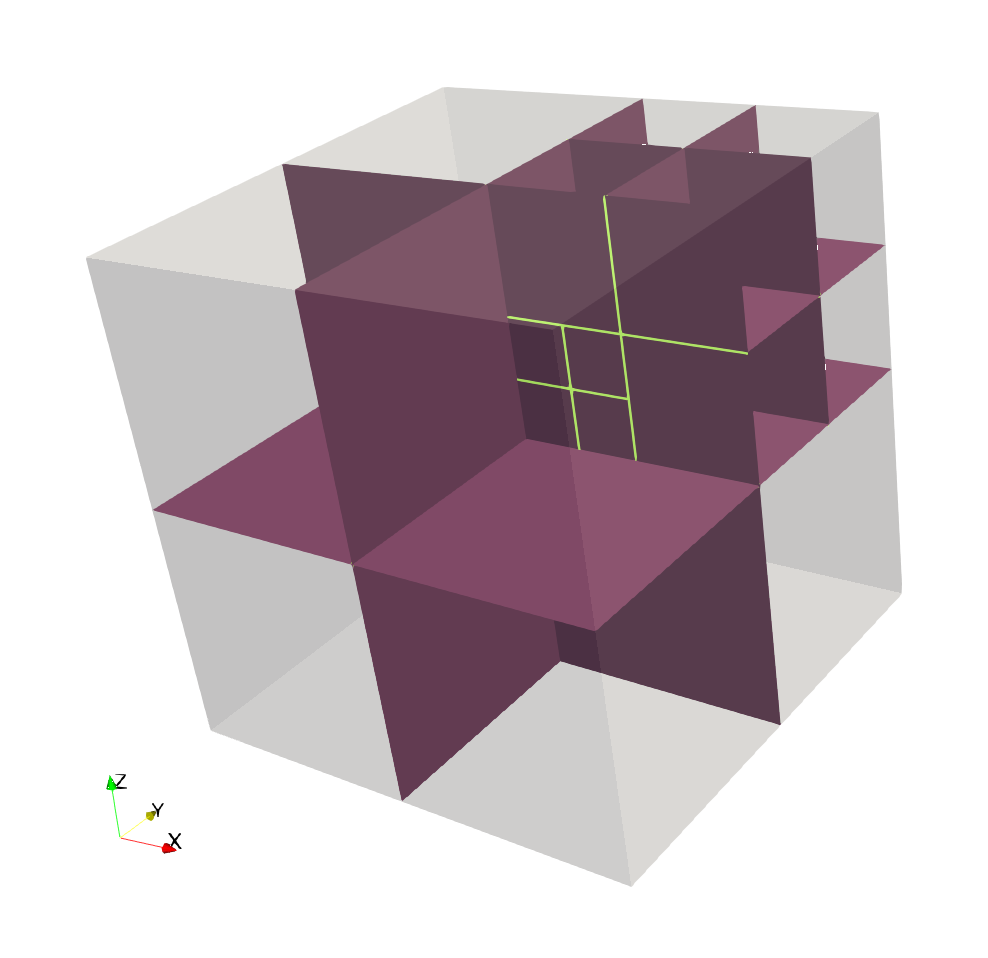}
	\includegraphics[width=0.45\textwidth]{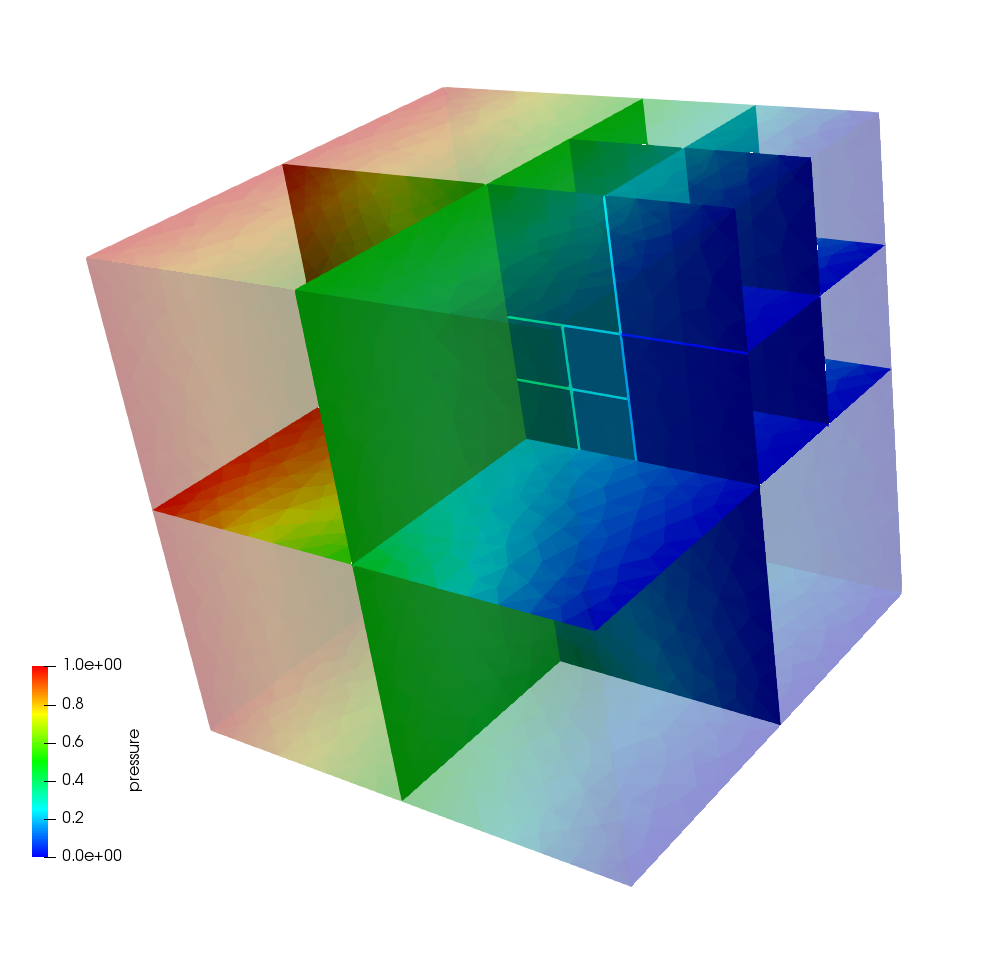}
	\caption{(Left) The three-dimensional unit cube domain in Example \ref{sub:example3} is decomposed by 9 fracture planes, 9 intersection lines and 1 intersection point. (Right) Pressure solution is presented for the case of a homogeneous permeability tensor $ \mathfrak{K} = \bm{I}$ and a mesh size $ h = 1/16 $.}
	\label{fig:geiger_3d}
\end{figure}

This example considers simulations of a 3D problem taken from the benchmark study \cite{berre2018}, that is, a three-dimensional Geiger fracture network. The rock domain is a unit cube intersected with a fracture network that consists of nine intersecting planes. The physical parameters are set as following: we take the fracture aperture to be $ 10^{-2} $ and the mixed-dimensional permeability tensor is homogeneous $ \mathfrak{K} = \bm{I} $. Within $ \mathfrak{K} $, we take into account that due to the reduced model scaling, the tangential $ K $ and the normal component $ K_{\nu} $ represent the effective values of the permeability field.  See \cite{berre2018} for more details. Furthermore, in the heterogeneous case, we consider splitting the tangential permeability into the rock matrix permeability $ K_m $ and fracture permeability $ K_f $ to allow for different flow patterns within the fracture network, either conducting or blocking the flow in the tangential direction. Also, we consider higher or lower normal permeability $ K_{\nu} $ that conducts or blocks the flow over the interface between the rock and the fractures. At the boundary, we impose pressure boundary conditions with unitary pressure drop from $ x = 0 $ to $ x = 1 $ boundary planes. The boundary conditions are applied to both the rock matrix and the fracture network. A graphical illustration of the geometry and the numerical solution is given in Figure \ref{fig:geiger_3d}.

Our goal is to investigate the robustness of the block preconditioners in \Cref{sub:md_preconditioner}. with respect to discretization parameter $ h $ and physical parameter  $ \mathfrak{K} $. We also vary the scaling parameter $ \alpha $ to study the influence on the convergence rate of the solver and how it changes with the heterogeneous permeability field. We compute and compare number of iterations of the outer and inner solver, as well as the elapsed process (CPU) time of the solver with regards to the number of degrees of freedom.

\begin{table}[htbp]
	\centering
	\scalebox{0.95}{
	\begin{tabular}{l|r||r|r|r||r|r|r||r|r|r}
		\hline \hline
		& \multicolumn{1}{c||}{} & \multicolumn{3}{c||}{$ \mathcal{M}_D $} & \multicolumn{3}{c||}{$ \mathcal{M}_L $} & \multicolumn{3}{c}{$ \mathcal{M}_U $} \\
		\cline{3-11} 
		$ h $ 		& $N_{dof}$ 	& $N_{it}$  & $T_{cpu}$ & rate  	& $N_{it}$	& $T_{cpu}$ & rate  & $N_{it}$ 	& $T_{cpu}$ & rate 		\\ \hline \hline
		$ 1/4 $ 	& 7173			& 12 \, (5) & 0.331   	& -- 		& 20 \, (5)	& 0.402   	& -- 	& 20 (4) 	& 0.351   	& -- 		\\ \hline
		$ 1/8 $ 	& 17172 		& 11 \, (6) & 0.580   	& 0.643 	& 19 \, (5)	& 0.617   	& 0.492 & 19 (5)	& 0.553   	& 0.523 	\\ \hline
		$ 1/16 $	& 89731			& 11 \, (6) & 3.229   	& 1.039 	& 19 \, (7)	& 4.265   	& 1.169	& 20 (5) 	& 3.716   	& 1.152		\\ \hline
		$ 1/32 $ 	& 518291		& 11 \, (8)	& 31.569 	& 1.300  	& 17 \, (8) & 39.499  	& 1.269	& 18 (7) 	& 37.431  	& 1.317		\\ \hline
		$ 1/64 $ 	& 3375415		& 11 (11) 	& 356.098	& 1.293	    & 17 (11)	& 482.206 	& 1.335	& 18 (9)	& 436.261 	& 1.311		\\ \hline
	\end{tabular}
	}
	\caption{Performance of the outer FGMRES solver using preconditioners $ \mathcal{M}_D $, $ \mathcal{M}_L $ and $ \mathcal{M}_U $ in \Cref{sub:example3} with regards to mesh refinement. For each preconditioner, we report number of outer (average inner) iterations $ N_{it} $ needed to reach the prescribed tolerance and overall elapsed CPU time $ T_{cpu} $. Last column presents the exponential rate of $ T_{cpu} $ of outer solver with regards to total degrees of freedom $ N_{dof} $. The permeability tensor is homogeneous and set to $ \mathfrak{K} = \bm{I} $ and the scaling parameter is set to $ \alpha = 1 $.}%
	\label{tab:example3_h}
\end{table}

The following tables consider the homogeneous permeability case with $ \mathfrak{K} = \bm{I} $ and $ \alpha = 1 $. In \Cref{tab:example3_h}, we present the results to study the robustness of the preconditioners with respect to the mesh refinement, where each row stands for a mesh twice finer than the previous one. For each preconditioner $ \mathcal{M}_D $, $ \mathcal{M}_L $ and $ \mathcal{M}_U $ we give the number of iterations $N_{it}$ of the outer FGMRES solver followed by average number of iterations of the inner GMRES solver in brackets, as well as the CPU time $T_{cpu}$ of the solving process and the exponential rate of the CPU time with regards to the number of degrees of freedom $ N_{dof} $. We clearly see that all preconditioners show that the number of iterations of the outer solver stays stable when refining the mesh, while there is a slight increase of iterations in the inner solver, which is due to our choice of $\mathfrak{B}^2_{reg}$ and $\mathfrak{B}^1_{reg}$. Therefore, we can conclude that preconditioners are robust with regards to the mesh size $ h $, but it suggests a different choice of the inner solver. As mentioned before, the inner solver performance depends on the choices of the spectrally equivalent approximations $\mathfrak{B}^2_{reg}$ and $\mathfrak{B}^1_{reg}$ of operators $ (\mathfrak{A}^2_{reg})^{-1} $ and $ (\mathfrak{A}^1_{reg})^{-1}$, respectively. These operators are represented in the nodal basis giving a Laplacian-type structure and thus, we have chosen UA-AMG as the approximation method. However, a further analysis that this UA-AMG approximation is actually spectrally equivalent is needed. Although the operators $ \mathfrak{A}^2_{reg} $ and $ \mathfrak{A}^1_{reg} $ act as a vector-Laplacian on each subdomain, they are still mixed-dimensional, and the off-diagonal coupling between the subdomains is still present which possibly diminishes the preferable structure for AMG methods. Moreover, this suboptimal behavior can be seen in the exponential rates of the CPU time $T_{cpu}$ of the total solving process with regards to the total number of degrees of freedom $ N_{dof} $. We expect $T_{cpu}$ to scale as $\mathcal{O}(N_{dof})$, giving a rate $ \approx 1 $, but all preconditioners show rate closer to $ 1.3 $. This is also visible in \Cref{fig:cpu_time} where the increase in $T_{cpu}$ fairly follows, but does not match the linear rate line. Even with a suboptimal process time performance, we still believe the preconditioners to be working well on the given problem setup, and consider the investigating proper spectrally equivalent approximations of operators $ (\mathfrak{A}^2_{reg})^{-1} $ and $ (\mathfrak{A}^1_{reg})^{-1} $ in future research.

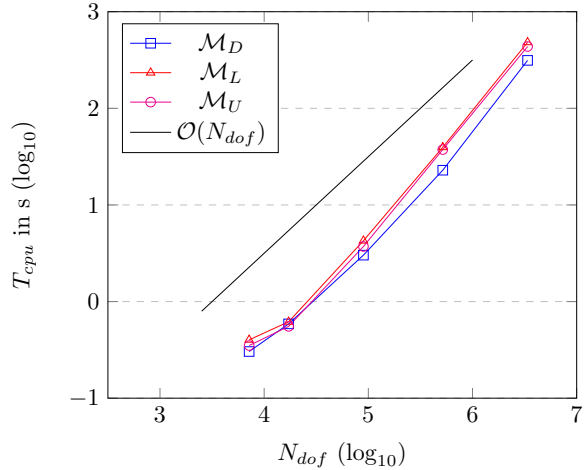
\begin{figure}[htbp]
	\centering 
	\begin{tikzpicture}[thick, scale=0.9]
	\begin{axis}[
	xlabel={$ N_{dof} $ ($\log_{10}$)},
	ylabel={$ T_{cpu} $ in s ($\log_{10}$)},
	xmin=2.5, xmax=7,
	ymin=-1, ymax=3,
	xtick={3, 4, 5, 6, 7},
	ytick={-1, 0, 1, 2, 3},
	legend pos=north west,
	ymajorgrids=true,
	grid style=dashed,
	]
	
	\addplot[
	color=blue,
	mark=square,
	]
	coordinates {
		(3.8557, -0.5171)(4.2348, -0.231)(4.9529, 0.4803)(5.7146, 1.3595)(6.5287, 2.4962)
	};
	\addlegendentry{$\mathcal{M}_D$}
	
	\addplot[
	color=red,
	mark=triangle,
	]
	coordinates {
		(3.8557, -0.3960)(4.2348, -0.2094)(4.9529, 0.6299)(5.7146, 1.5966)(6.5287, 2.6832)
	};
	\addlegendentry{$\mathcal{M}_L$}
	
	\addplot[
	color=magenta,
	mark=o,
	]
	coordinates {
		(3.8557, -0.4552)(4.2348, -0.2570)(4.9529, 0.5700)(5.7146, 1.5732)(6.5287, 2.6397)
	};
	\addlegendentry{$\mathcal{M}_U$}
	
	\addplot[
	domain=3.4:6,
	samples=40,
	color=black,
	]
	{(x - 3.5)};
	\addlegendentry{$\mathcal{O}(N_{dof})$}
	
	\end{axis}
	\end{tikzpicture}
	\caption{CPU time $ T_{cpu} $ of FGMRES solver with block preconditioners compared to total number of degrees of freedom $ N_{dof} $ of the linear system in \Cref{sub:example3}. The values of $ T_{cpu} $ and $ N_{dof} $ are taken from \Cref{tab:example3_h}. We mark $\mathcal{O}(N_{dof})$ complexity with a black continuous line.}
	\label{fig:cpu_time}
\end{figure}

\begin{table}[htbp]
	\centering
	\begin{tabular}{r|r|r|r}
		\hline \hline
		\multicolumn{1}{c||}{$\alpha$} & $\mathcal{M}_D$ & $\mathcal{M}_L$ & $\mathcal{M}_U$ \\ \hline \hline
		\multicolumn{1}{r||}{$10^0$} 		& 11 \, (8) & 17 \, (8)	& 18 \, (7)  \\ \hline
		\multicolumn{1}{r||}{$10^1$} 		& 6 \, (9) 	& 9 \, (9)	& 10 \, (8)   \\ \hline
		\multicolumn{1}{r||}{$10^2$} 	& 5 (10) 	& 7 (10)	& 7 \, (8)    \\ \hline
		\multicolumn{1}{r||}{$10^3$} 	& 4 (12) 	& 5 (11)	& 7 \, (9)     \\ \hline
		\multicolumn{1}{r||}{$10^4$} 	& 4 (13) 	& 4 (13)	& 6 (10)    \\ \hline
	\end{tabular}
	\caption{Performance of the outer FGMRES solver using preconditioners $ \mathcal{M}_D $, $ \mathcal{M}_L $ and $ \mathcal{M}_U $ in \Cref{sub:example3} with regards to varying the scaling parameter $ \alpha $. For each preconditioner, we report number of outer (average inner) iterations needed to reach the prescribed tolerance. The permeability tensor is homogeneous and set to $ \mathfrak{K} = \bm{I} $ and mesh size is set to $ h = 1/32 $.}%
	\label{tab:example3_results_alpha}
\end{table}

While still taking the permeability tensor to be homogeneous and unitary, we set the mesh size to $ h = 1/32 $ and study the performance of the preconditioners with a range of values of the parameter $ \alpha $. Although the theory suggests taking any $ \alpha \geq \mathfrak{K}_{min}^{-1} $, we consider instead $ \alpha \geq \max \{ 1, \mathfrak{K}_{min}^{-1} \} $ to achieve reasonable convergence of the underlying augmented Lagrangian method. \Cref{tab:example3_results_alpha} shows the results of the overall outer (and average inner) number of iterations for both diagonal and triangular preconditioners. As expected, the performance of block preconditioners improves with higher values of $ \alpha $ since, according to the theory of the augmented Lagrangian method~\cite{fortin2000augmented}, the iterative method should converge faster in those cases. On the other hand, increasing $ \alpha $ gives more weight on the mixed-dimensional divergence part of the inner product \eqref{eq:flux-block}, which makes the problem at each inner iteration nearly singular \cite{tuminaro:xu:zhu, lee:wu:xu:zikatanov}. This may slightly deteriorate the performance of the inner GMRES method, that mostly affects the UA-AMG method within it. Nevertheless, we find a good balance to performance of both the outer and inner solver to be around $ \alpha = \max \{1, 100 \mathfrak{K}_{min}^{-1} \} $. This can be observed in the study on the heterogeneous permeability field in \Cref{tab:example3_results_alpha_K}. Here, we set the tangential rock component of the permeability to be $ K_m = \bm{I} $, while the tangential fracture component $ K_f $ and the normal fracture component $ K_{\nu} $ in conjunction assume different values, from low to high permeable case. The results show similar behavior as in \Cref{tab:example3_results_alpha}: we get a lower number of outer iterations for  $ \alpha \gg \max \{ 1, \mathfrak{K}_{min}^{-1} \} $, but in turn the inner number of iteration increases. Therefore, in this example, we can conclude that taking $ \alpha = \max \{1, 100 \mathfrak{K}_{min}^{-1} \} $ gives the optimal performance of the preconditioned iterative method.

\begin{table}[htbp]
	\centering
	\begin{tabular}{r|r|r|r|r|r}
		\hline \hline
		\multicolumn{1}{c||}{}
		& \multicolumn{5}{c}{$ K_f = K_{\nu} $} \\
		\cline{2-6}
		\multicolumn{1}{c||}{$\alpha$} 	& $10^{-4}$ & $10^{-2}$ & $10^{0}$ 	& $10^{2}$ 	& $10^{4}$ 	\\ \hline \hline
		\multicolumn{1}{l||}{$10^{0}$} 	& --  		& -- 		& 11 \, (5)	& 5	\, (5)	& 5 (13) 	\\ \hline
		\multicolumn{1}{l||}{$10^{2}$} 	& -- 		& 11 (5) 	& 5  \, (6) & 4 \, (9)	& 4 (14) 	\\ \hline
		\multicolumn{1}{l||}{$10^{4}$} 	& 11 (5)  	& 5 (6)		& 4	 (15)	& 4 (22)	& 4 (41) 	\\ \hline
	\end{tabular}
	\caption{Performance of the outer FGMRES solver using preconditioners $ \mathcal{M}_D $, $ \mathcal{M}_L $ and $ \mathcal{M}_U $ in \Cref{sub:example3} with regards to varying the scaling parameter $ \alpha $ and the lowest eigenvalue of the permeability tensor $ \mathfrak{K}_{min} $. The variations in the eigenvalue spectrum come from the heterogeneity of the fractured porous medium: the tangential rock component of the permeability is $ K_m = \bm{I} $, while we vary the tangential fracture component $ K_f $ and the normal fracture component $ K_{\nu} $. For each preconditioner, we report number of outer (average inner) iterations needed to reach the prescribed tolerance. The mesh size is set to $ h = 1/16 $.}%
	\label{tab:example3_results_alpha_K}
\end{table}

\subsection{Example: Two-dimensional Complex Network}
\label{sub:example2}

\begin{figure}[htbp]
	\centering
	\includegraphics[valign=b, width=0.435\textwidth]{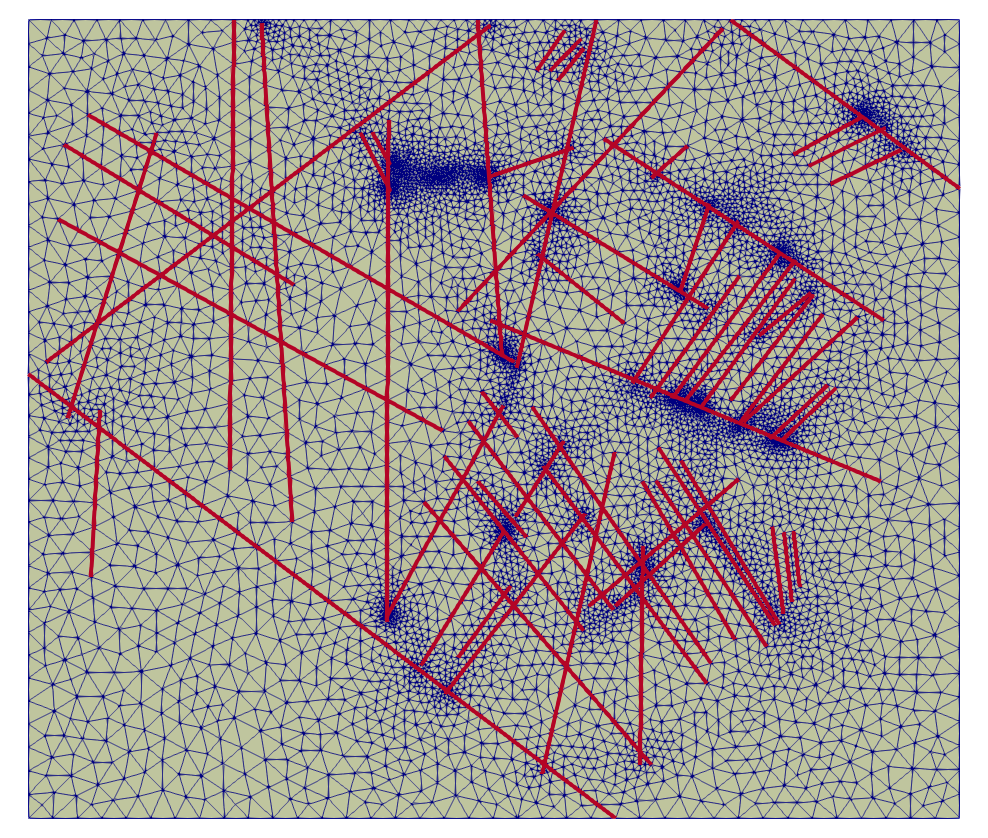}
	\includegraphics[valign=b, width=0.485\textwidth]{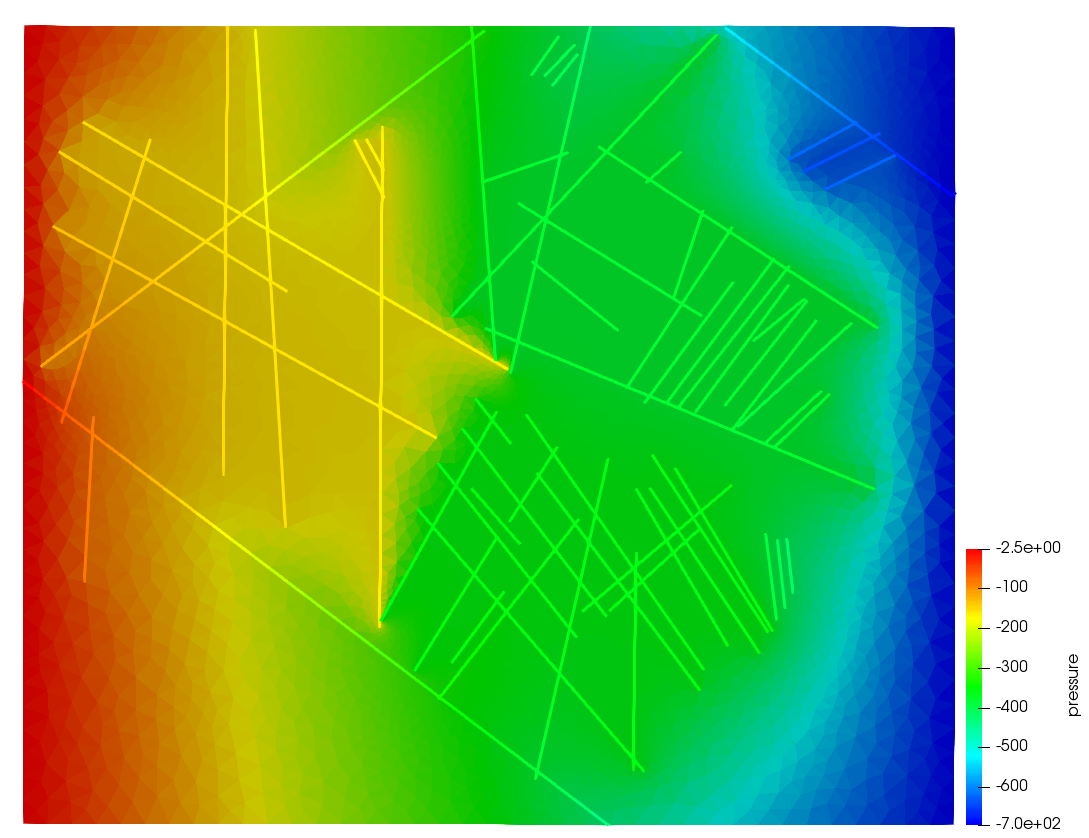}
	\caption{(Left) Graphical representation of the two-dimensional domain and fracture network
		geometry of Example \ref{sub:example2}. (Right) Pressure solution profile.}
	\label{fig:sotra_2d_domain}
\end{figure}

We provide another example, chosen from the benchmark study \cite{Flemisch2016a}, which contains a fracture network from an interpreted outcrop in the Sotra island, near Bergen, Norway. The network includes 63 fractures, all with different length. The porous medium domain spatial dimensions are $ 700 $ m $ \times $ $ 600 $ m with uniform matrix permeability $ K_m = \bm{I} $ m$ ^2 $.  All  the  fractures  have  the  same  scalar tangential and normal permeability $ K_n = K_f = 10^5 \bm{I} $ m$ ^2 $ and aperture $ \varepsilon = 10^{-2} $ m. The permeability tensors $ K_n $ and $ K_f $ are considered to be the effective values, meaning that we incorporate the aperture scaling with $ \varepsilon $ within the permeability values due to the reduced fracture modeling. See \cite{Flemisch2016a, boon:flow2018} for the detailed description of the scalings. The pressure boundary conditions are imposed on all boundaries, with a linear unitary pressure drop from the left to the right boundary. Throughout all the tests, we use a fixed mesh grid with a typical mesh size $ h = 18.75 $ m and total of 44765 degrees of freedom. See \Cref{fig:sotra_2d_domain} for an illustration of the domain, the mesh and the numerical solution of this problem.

This more realistic case of a fracture network is chosen to demonstrate the robustness of our auxiliary preconditioners, even with a larger number of fractures in the system and a complex fracture network configuration. Large-scale simulations often require handling those features of fractured porous media, as they often appear in geological rock formations in the subsurface and can significantly influence the stability of the any given solving method. In this case, the sharp tips and very acute intersections of fractures may decrease the shape regularity of the mesh, but also increase the condition number of the system and the number of unknowns, as seen in this example and \Cref{fig:sotra_2d_domain}. Therefore, we aim to show that our preconditioners still show a good performance under these challenging conditions.

\begin{table}[htbp]
	\centering
	\begin{tabular}{r||r|r|r}
		\hline \hline
		$ \alpha $	& $\mathcal{M}_D$ 	& $\mathcal{M}_L$ 	& $\mathcal{M}_U$ 	\\ \hline \hline
		$ 10^2 $  	& 40 (13)  			& 78 (10)			& 79 (9)			\\ \hline
		$ 10^3 $   	& 15 \, (9)			& 24 \, (9)			& 25 (8)			\\ \hline
		$ 10^4 $ 	& 8 \, (5) 			& 10 \, (5) 		& 11 (5) 			\\ \hline
		$ 10^5 $ 	& 5 \, (4) 			& 8 \, (5)			& 4 (5) 			\\ \hline
		$ 10^6 $ 	& 6 (28)			& 7 (11)			& 12  (4) 			\\ \hline
	\end{tabular}
	\caption{Performance of the outer FGMRES solver using preconditioners $ \mathcal{M}_D $, $ \mathcal{M}_L $ and $ \mathcal{M}_U $ in \Cref{sub:example2} with regards to varying the scaling parameter $ \alpha $. For each preconditioner, we report number of outer (average inner) iterations needed to reach the prescribed tolerance. The permeability tensors are set to $ K_m = \bm{I} $, $ K_n = K_f = 10^{5} \bm{I} $ and the mesh size is set to $ h = 18.75 $.}%
	\label{tab:example2_alpha}
\end{table}

We first consider different values of the parameter $ \alpha $, with results given in \Cref{tab:example2_alpha}. As before, the performance of the diagonal $\mathcal{M}_D$ and triangular preconditioners $\mathcal{M}_L$ and $\mathcal{M}_U$ improves with larger values of $ \alpha $, reaching relatively optimal value at $ \alpha = 10^5 $ for all three preconditioners. This is different from the previous example in \Cref{sub:example3} where the best results are given when $ \alpha = \max \{ 1,   100 \mathfrak{K}_{min}^{-1} \} $, considering that in this case we have $ \mathfrak{K}_{min}^{-1} = 10^{-5} $. However, there are many differences in the problem settings of these two examples that need to be taken into consideration. First, according to \Cref{thm:condition_number_Bk} ,the performance of the mixed-dimensional auxiliary space preconditioners can depend on the mixed-dimensional domain $ \Omega $ and the regularity of the corresponding mesh. In comparison to the example in \Cref{sub:example3}, the ambient domain in this example is two-dimensional, the domain is more rectangular-type and, due to the complex fracture network configuration, the mesh is less regular. Therefore, we expect a different behavior of both the outer FGMRES and inner GMRES solver in this example. Particularly, this can be seen in \Cref{tab:example2_alpha}, where the number of outer and inner iterations reduces for larger values of the scaling parameter $ \alpha $, though it started with a large number of iterations in all preconditioners for $ \alpha = 10^2 $, and for $ \alpha = 10^6 $ it get slightly larger again. We remind that although larger values of parameter $ \alpha $ should improve the performance of the block preconditioners, the divergence part of the inner product \eqref{eq:flux-block} now dominates, which makes it harder for the inner solver to convergence because of the problem becomes more nearly singular \cite{tuminaro:xu:zhu, lee:wu:xu:zikatanov}. 

\begin{table}[htbp]
	\centering
	\begin{tabular}{r||r|r||r|r||r|r||r|r}
		\hline
		\multicolumn{1}{c||}{}
		& \multicolumn{8}{c}{$ \mathcal{M}_D $} \\
		\cline{2-9}
		\multicolumn{1}{c||}{}
		& \multicolumn{2}{c||}{Case 1} & \multicolumn{2}{c||}{Case 2} & \multicolumn{2}{c||}{Case 3} & \multicolumn{2}{c}{Case 4} \\
		\cline{2-9}
		\cline{2-9} \cline{2-9}
		$ N_{fracs} $ 	& $ N_{dof} $ 	& $ N_{it}$ & $ N_{dof} $	& $ N_{it}$	& $ N_{dof} $ 	& $ N_{it}$	& $ N_{dof} $ 	& $ N_{it}$	\\ \hline \hline
		1  				& 8241  		& 6 (3)  	& 8101			& 7 (2)		& 8891			& 6 (3)		& 8561			& 7 (3) 	\\ \hline
		5  				& 17661   		& 7 (3)		& 10838			& 6 (3)		& 9300			& 6 (3)		& 11751			& 6 (3)		\\ \hline
		10 				& 15809   		& 6 (3) 	& 14437			& 7 (3) 	& 9180			& 6 (3) 	& 11998			& 6 (3)		\\ \hline
		20 				& 23083  		& 7 (3) 	& 19147			& 6 (4)		& 13659			& 7 (3) 	& 17341			& 6 (4)		\\ \hline
		40 				& 31295 		& 5 (4)		& 25980			& 6 (4)		& 29032			& 7 (4) 	& 27654			& 6 (4)		\\ \hline
		63 				& 44765 		& 5 (4)		& 44765			& 5 (4)		& 44765			& 5 (4) 	& 44765			& 5 (4)		\\ \hline
	\end{tabular}
	\caption{Performance of the outer FGMRES solver using preconditioners $ \mathcal{M}_D $, $ \mathcal{M}_L $ and $ \mathcal{M}_U $ in \Cref{sub:example2} with regards to varying number of fractures $ N_{fracs} $ in the fracture network. For each preconditioner, we report number of outer (average inner) iterations $ N_{it} $ needed to reach the prescribed tolerance. The permeability tensors are set to $ K_m = \bm{I} $, $ K_n = K_f = 10^{5} \bm{I} $, the mesh size is set to $ h = 18.75 $ and the scaling parameter $ \alpha = 10^5 $.}%
	\label{tab:example2_frac_robust}
\end{table}

It is not only the case that the fracture network is more complex, we also have many more fractures included in the domain. This factor should not affect the performance of the preconditioners, which we aim to show in the next set of numerical tests. In the following, we only test the block diagonal preconditioner $ \mathcal{M}_D $ since it shows overall best behavior in comparison to the block triangular ones, this particularly evident from \Cref{tab:example2_alpha}. We also set the scaling parameter $ \alpha = 10^5 $.

We consider different numbers of fractures included in the original fracture network of 63 fractures in \Cref{fig:sotra_2d_domain}. To this end, we randomly select and gradually add more fractures to the network, starting from 1 fracture, to 5, 10, 20, 40 and ultimately all 63 fractures included. We repeat the process four times, creating four different cases, each having either 1, 5, 10, 20 or 40 fractures. See \Cref{fig:four_cases_example_2} for an illustration of pressure solutions to all four cases, each with 20 randomly selected fractures. The reason to constructing four cases is to eliminate bias in selecting fractures in specific order. We report in \Cref{tab:example2_frac_robust}, for all four cases, the number of degrees of freedom $ N_{dof} $ and the number of outer (average inner) iterations $ N_{it} $ of the FGMRES (GMRES) method preconditioned with the diagonal preconditioner $ \mathcal{M}_D $. It is clear that the preconditioned outer iterative method does not depend on the number of fractures in the fracture network, in all the cases. The same can be seen in the inner solver showing a relatively even number of iterations. Therefore, the robustness of the preconditioner $ \mathcal{M}_D $ with regards to the number of fractures in the fracture network is shown, which is consistent with the analysis in the previous sections. 

\begin{figure}
	\begin{tabular}{cc}
		\includegraphics[width=65mm]{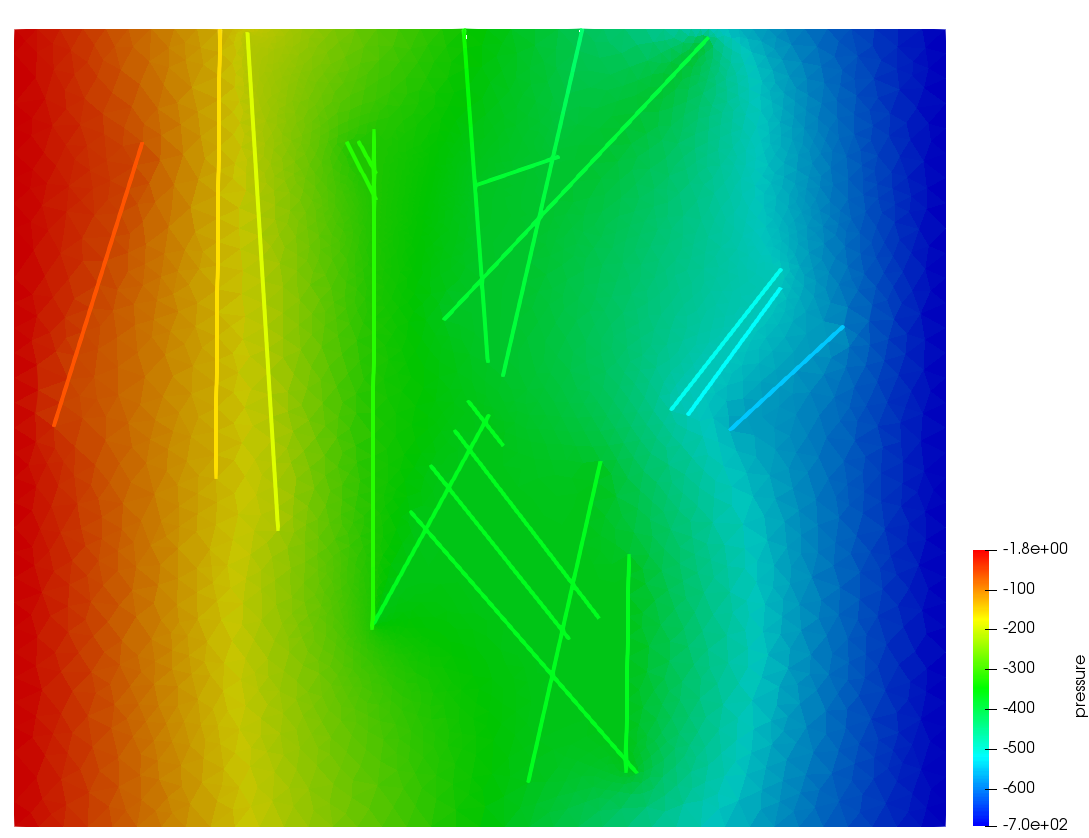} &   \includegraphics[width=65mm]{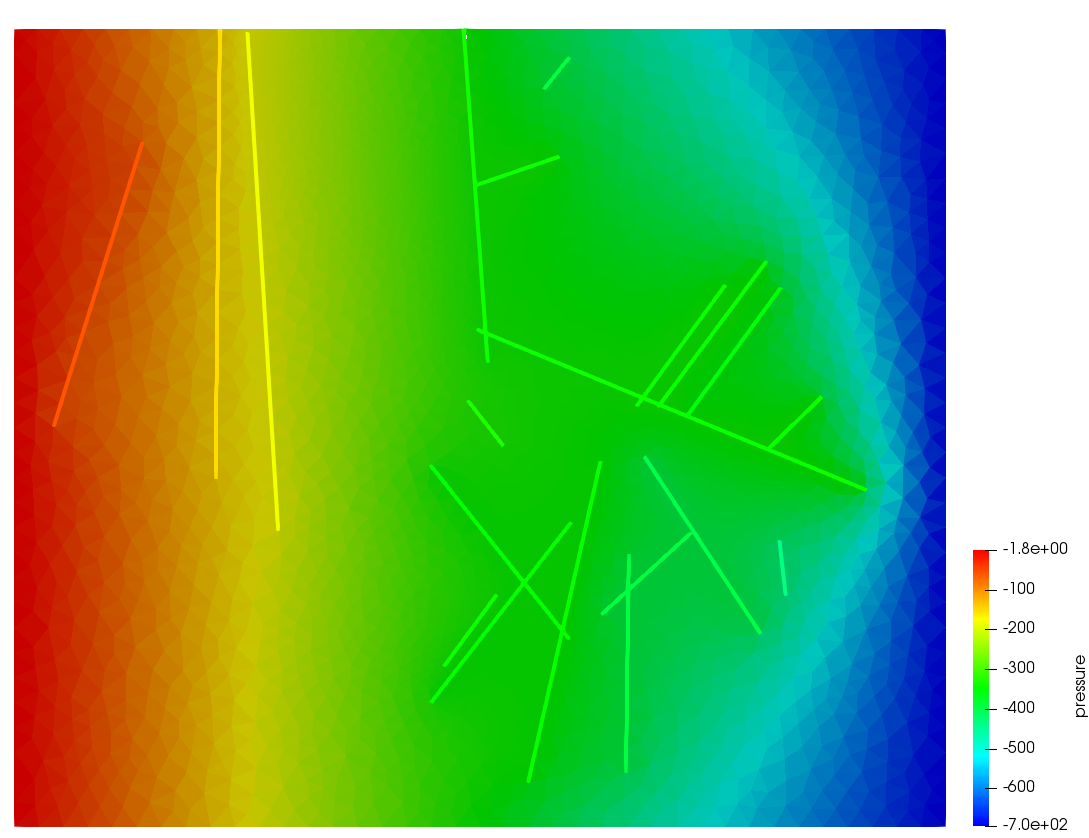} \\
		(a) Case 1 & (b) Case 2 \\[6pt]
		\includegraphics[width=65mm]{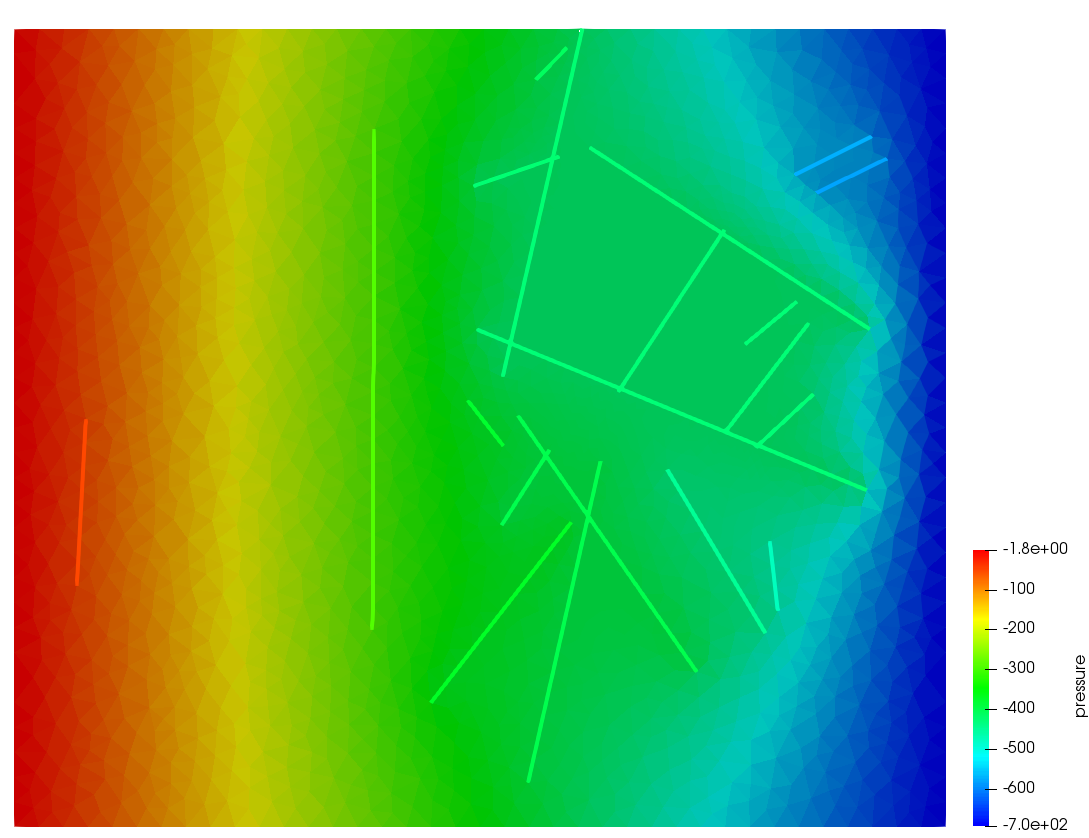} &   \includegraphics[width=65mm]{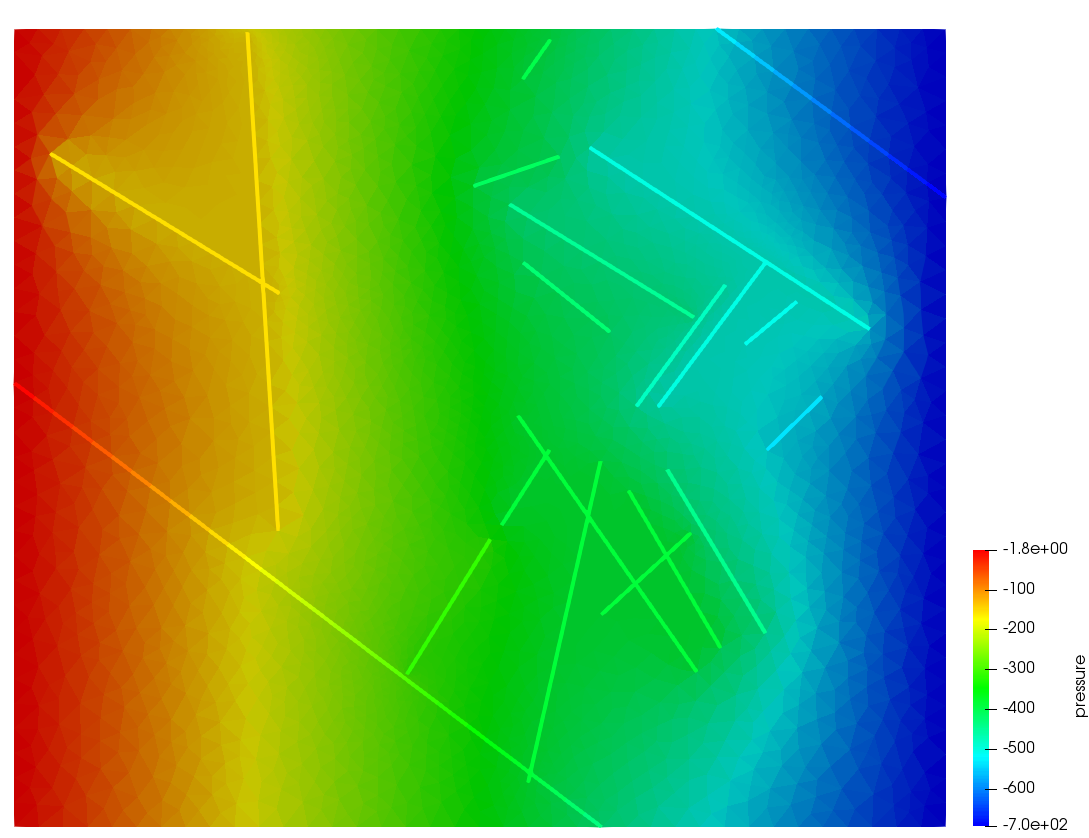} \\
		(c) Case 3 & (d) Case 4 \\[6pt]
	\end{tabular}
	\caption{Graphical representation of the pressure solution in Example \ref{sub:example2} with only 20 fractures included. Each Case 1--4 corresponds to selecting randomly 20 fractures from the original set of 63 fractures, without changing the original positioning of the selected fractures within the domain.}
	\label{fig:four_cases_example_2}
\end{figure}

\section{Conclusion} \label{sec:conclusion}

In this work, we have derived nodal auxiliary space preconditioners for discretizations of mixed-dimensional partial differential equations. In order to do so, we have extended the stable regular decomposition, both in continuous and discrete setting, to mixed-dimensional geometries. The resulting decomposition differs from the fixed-dimensional case in the way that we do not consider directly the regular inverse, but we establish the regular decomposition hierarchically by combining the regular decompositions on each sub-manifold of the mixed-dimensional domain. Based on this and the auxiliary space preconditioning framework, we propose robust preconditioners to solving mixed-dimensional elliptic problems. We demonstrate how these preconditioners are derived and implemented with an example of mixed-dimensional model of flow in fractured porous media. The robustness of the preconditioners is also verified of two benchmark numerical experiments of fractured porous media.  From the numerical experiments, we also see the need of a robust method for solving Laplacian problem in the mixed-dimensional setting in order to further improve the robustness and effectiveness of the proposed preconditioners.  This is the topic for our future work. 


\section{Acknowledgments}
The first author acknowledges the financial support from the TheMSES project funded by Norwegian Research Council grant 250223. The second author thanks the Deutsche Forschungsgemeinschaft (DFG, German Research Foundation) for supporting this work by funding SFB 1313, Project Number 327154368. The work of the third author is partially supported by the National Science Foundation under grant DMS-1620063. 
A special thanks is extended to Jan M. Nordbotten for valuable comments and discussions on the presented work.


\bibliographystyle{siam}
\bibliography{references}

\end{document}